\documentclass[11pt, reqno]{amsart}
\usepackage{amssymb,latexsym,amsmath,amsfonts}
\usepackage{latexsym}
\usepackage[mathscr]{eucal}
\usepackage{bm}

\def ~{\hspace{1mm}}

\def\textmatrix#1&#2\\#3&#4\\{\bigl({#1 \atop #3}\ {#2 \atop #4}\bigr)}
\def\dispmatrix#1&#2\\#3&#4\\{\left({#1 \atop #3}\ {#2 \atop #4}\right)}
\newcommand{\beg}{\begin{equation}}
\newcommand{\eeg}{\end{equation}}
\newcommand{\ben}{\begin{eqnarray*}}
\newcommand{\een}{\end{eqnarray*}}

\newtheorem{thm}{Theorem}[section]
\newtheorem{cor}[thm]{Corollary}
\newtheorem{lem}[thm]{Lemma}

\newtheorem{prop}[thm]{Proposition}
\numberwithin{equation}{section}
\theoremstyle{definition}
\newtheorem{defn}[thm]{Definition}
\newtheorem{rem}[thm]{Remark}

\newtheorem{eg}[thm]{Example}

\def\textmatrix#1&#2\\#3&#4\\{\bigl({#1 \atop #3}\ {#2 \atop #4}\bigr)}
\def\dispmatrix#1&#2\\#3&#4\\{\left({#1 \atop #3}\ {#2 \atop #4}\right)}

\begin{document}
\title[Subvarieties of the tetrablock]
{Subvarieties of the tetrablock and von Neumann's inequality}

\author[Sourav Pal]{Sourav Pal}
\address[Sourav Pal]{Mathematics Department, Indian Institute of Technology Bombay,
Powai, Mumbai - 400076, India.} \email{sourav@math.iitb.ac.in,
sourav@isid.ac.in}



\thanks{The author is supported by Seed Grant of IIT Bombay, CPDA and the INSPIRE
Faculty Award (Award No. DST/INSPIRE/04/2014/001462) of DST,
India.}

\begin{abstract}
We show an interplay between the complex geometry of the
tetrablock $\mathbb E$ and the commuting triples of operators
having $\overline{\mathbb E}$ as a spectral set. We prove that
$\overline{\mathbb E}$ being a $3$-dimensional domain does not
have any $2$-dimensional distinguished variety, every
distinguished variety in the tetrablock is one-dimensional and can
be represented as
\begin{equation}\label{eqn:1}
\Omega=\{ (x_1,x_2,x_3)\in \mathbb E \,:\, (x_1,x_2) \in
\sigma_T(A_1^*+x_3A_2\,,\, A_2^*+x_3A_1) \},
\end{equation}
where $A_1,A_2$ are commuting matrices of the same order
satisfying\\ $[A_1^*,A_1]=[A_2^*,A_2]$ and a norm condition. The
converse also holds, i.e, a set of the form (\ref{eqn:1}) is
always a distinguished variety in $\mathbb E$. We show that for a
triple of commuting operators $\Upsilon = (T_1,T_2,T_3)$ having
$\overline{\mathbb E}$ as a spectral set, there is a
one-dimensional subvariety $\Omega_{\Upsilon}$ of
$\overline{\mathbb E}$ depending on $\Upsilon$ such that
von-Neumann's inequality holds, i.e,
\[
f(T_1,T_2,T_3)\leq \sup_{(x_1,x_2,x_3)\in\Omega_{\Upsilon}}\,
|f(x_1,x_2,x_3)|,
\]
for any holomorphic polynomial $f$ in three variables, provided
that $T_3^n\rightarrow 0$ strongly as $n\rightarrow \infty$. The
variety $\Omega_\Upsilon$ has been shown to have representation
like (\ref{eqn:1}), where $A_1,A_2$ are the unique solutions of
the operator equations
\begin{gather*}
T_1-T_2^*T_3=(I-T_3^*T_3)^{\frac{1}{2}}X_1(I-T_3^*T_3)^{\frac{1}{2}}
\text{ and } \\
T_2-T_1^*T_3=(I-T_3^*T_3)^{\frac{1}{2}}X_2(I-T_3^*T_3)^{\frac{1}{2}}.
\end{gather*}
We also show that under certain condition, $\Omega_{\Upsilon}$ is
a distinguished variety in $\mathbb E$. We produce an explicit
dilation and a concrete functional model for such a triple
$(T_1,T_2,T_3)$ in which the unique operators $A_1,A_2$ play the
main role. Also, we describe a connection of this theory with the
distinguished varieties in the bidisc and in the symmetrized
bidisc.
\end{abstract}

\maketitle

\tableofcontents

\section{Introduction}

\noindent The tetrablock is a polynomially convex, non-convex and
inhomogeneous domain in $\mathbb C^3$ defined by
\[
\mathbb E = \{ (x_1,x_2,x_3)\in\mathbb C^3\,:\,
1-zx_1-wx_2+zwx_3\neq 0 \textup{ whenever } |z|\leq 1, |w|\leq 1
\}.
\]
This domain has attracted a lot of attention of the function
theorists, complex geometers and operator theorists over past one
decade because of its connection with $\mu$-synthesis and
$H^{\infty}$ control theory
(\cite{awy,awy-cor,young,EZ,EKZ,Zwo,tirtha,tirtha-sau,sourav1}).
The \textit{distinguished boundary} of the tetrablock, which is
same as the $\check{\textup{S}}$ilov boundary, was determined in
\cite{awy} (Theorem 7.1 in \cite{awy}) to be the set
\begin{align}
b{\mathbb E} & = \{ (x_1,x_2,x_3)\in\mathbb C^3\,:\, x_1=\bar
x_2x_3,\,|x_3|=1 \textup{ and } |x_2|\leq 1 \} \notag
\\ \label{dv1}& = \{ (x_1,x_2,x_3)\in\overline{\mathbb E}: |x_3|=1 \}.
\end{align}
Amongst the characterizations given in \cite{awy} of the points in
$\mathbb E$, the following is the most elegant one that clarifies
the geometric location of $\mathbb E$. This will be used
frequently in all the sections of this paper.
\begin{thm}\label{thm1}
A point $(x_1,x_2,x_3)\in \mathbb C^3$ is in $\mathbb E$
(respectively in $\overline{\mathbb E}$) if and only if $|x_3|< 1$
(respectively $\leq 1$) and there exist $\beta_1,\beta_2 \in
\mathbb C$ such that $|\beta_1|+|\beta_2| <1$ (respectively $\leq
1$) and $x_1=\beta_1 + \bar{\beta_2}x_3, \quad
x_2=\beta_2+\bar{\beta_1}x_3$.
\end{thm}
It is evident from the above theorem that the tetrablock lives
inside the tridisc $\mathbb D^3$ and that the topological boundary
$\partial \mathbb E$ of the tetrablock is given by
\begin{align*}
\partial \mathbb E & =\{ (x_1,x_2,x_3)\in\overline{\mathbb E}:\, |x_3|=1
\}\\& \quad \cup \{ (x_1,x_2,x_3)\in\overline{\mathbb E}:\,
|x_3|<1, |\beta_1|+|\beta_2|=1 \}\\& = b\mathbb E \cup \{
(x_1,x_2,x_3)\in\overline{\mathbb E}:\, |x_3|<1,
|\beta_1|+|\beta_2|=1 \},
\end{align*}
where $\beta_1, \beta_2$ are as of Theorem \ref{thm1}.\\

A variety $V_S$ in $\mathbb C^n$, where $S$ is a set of
polynomials in $n$-variables $z_1,\dots,z_n$, is a subset of
$\mathbb C^n$ defined by
\[
V_S= \{ (z_1,\hdots,z_n)\in\mathbb C^n \,:\, p(z_1,\hdots,z_n)=0,
\quad \text{ for all } p\in S \}.
\]
A variety $W$ in a domain $G\subseteq \mathbb C^n$ is the part of
a variety lies inside $G$, i.e, $W=V_S\cap G$ for some set $S$ of
polynomials in $n$-variables. A {\em distinguished variety} in a
polynomially convex domain $G$ is a variety that intersects the
topological boundary of $G$ at its $\check{\textup{S}}$ilov
boundary. See \cite{AM05,AM06,Knese10,pal-shalit} to follow some
recent works on disitnguished varieties in the bidisc and in the
symmetrized bidisc. Therefore, in particular a distinguished
variety in the tetrablock is defined in the following way.
\begin{defn}
A set $\Omega \subseteq \mathbb E$ is said to be a distinguished
variety in the tetrablock if $\Omega$ is a variety in $\mathbb E$
such that
\begin{equation}\label{eqn01}
\overline{\Omega}\cap \partial \mathbb E = \overline{\Omega}\cap
b\overline{\mathbb E}.
\end{equation}
\end{defn}
We denote by $\partial \Omega$ the set described in (\ref{eqn01}).
It is evident from the definition that a distinguished variety in
$\mathbb E$ is a one or two-dimensional variety in $\mathbb E$
that exits the tetrablock through the distinguished boundary $b\mathbb E$.\\

The main aim of this paper is to build and explain a connection
between the complex geometry of the domain $\mathbb E$ and the
triple of commuting operators having $\overline{\mathbb E}$ as a
spectral set. The principal source of motivation for us is the
seminal paper \cite{AM05} of Agler and M$^{\textup{c}}$Carthy. In
our first main result Theorem \ref{thm:DVchar}, we show that no
distinguished variety in $\mathbb E$ can be $2$-dimensional, all
distinguished varieties in $\mathbb E$ have complex dimension one
and can be represented as
\begin{equation}\label{eqn1}
\{(x_1,x_2,x_3)\in\mathbb E\,:\, (x_1,x_2)\in
\sigma_T(A_1^*+x_3A_2, A_2^*+x_3A_1) \},
\end{equation}
where $A_1,A_2$ are commuting matrices of same order such that
$[A_1^*,A_1]=[A_2^*,A_2]$ and that $\|A_1^*+A_2z\|_{\infty,
\mathbb T}\leq 1$, $\mathbb T$ being the unit circle in the
complex plane and
\[
\|A_1^*+A_2z\|_{\infty, \mathbb T}=\sup_{z\in\mathbb
T}\|A_1^*+A_2z\|.
\]
Here $\sigma_T(S_1,\cdots,S_n)$ denotes the \textit{Taylor joint
spectrum} of a commuting $n$-tuple of operators $(S_1,\hdots,S_n)$
which consists of joint eigenvalues of $S_1,\cdots,S_n$ only when
they are matrices (see \cite{KLMV} for a detailed proof of this
fact). Since only the case $n=2$ is used here, we shall have a
brief discussion on Taylor joint spectrum of a pair of commuting
bounded operators at the beginning of section 2. Also $[S_1,S_2]$
denotes the commutator $S_1S_2-S_2S_1$. Conversely, every subset
of the form (\ref{eqn1}) is a distinguished variety in $\mathbb E$
provided that $\|A_1^*+A_2z\|_{\infty, \mathbb T}<1$. Examples
show that a set of this kind may or may not be a distinguished
variety if $\|A_1^*+A_2z\|_{\infty, \mathbb T}=1$. It is
surprising that the tetrablock, being a domain of complex
dimension $3$, does not have a two-dimensional distinguished
variety. Thus, the study of the distinguished varieties in
$\mathbb E$ leads us to operator theory, in particular to matrix
theory.

On the other hand, another main result, Theorem \ref{thm:VN},
shows that if a triple of commuting operators
$\Upsilon=(T_1,T_2,T_3)$, defined on a Hilbert space $\mathcal H$,
has $\overline{\mathbb E}$ as a spectral set and if the
fundamental operators $A_1,A_2$ of $\Upsilon$ are commuting
matrices satisfying $[A_1^*,A_1]=[A_2^*,A_2]$, then there is a
one-dimensional subvariety $\Omega_\Upsilon$ of $\overline{\mathbb
E}$ such that $\Omega_\Upsilon$ is a spectral set for $\Upsilon$.
The fundamental operators $A_1,A_2$ are two unique operators
associated with such a commuting triple $(T_1,T_2,T_3)$ and are
explained below. Also $\omega (T)$ denotes the numerical radius of
an operator $T$. The one-dimensional subvariety $\Omega_\Upsilon$
is obtained in terms of the fundamental operators $A_1,A_2$ in the
following way,
\[
\Omega_\Upsilon =\{ (x_1,x_2,x_3)\in\overline{\mathbb E}\,:\,
(x_1,x_2)\in\sigma_T(A_1^*+x_3A_2,A_2^*+x_3A_1) \}.
\]
 Moreover, when $A_1,A_2$ satisfy the condition
$\omega(A_1+A_2z)<1$, for all $z\in\mathbb T$, $\Omega_\Upsilon
\cap \mathbb E$ is a distinguished variety in $\mathbb E$. So, it
is remarkable that for such an operator triple $\Upsilon$ one can
extract a one-dimensional curve from ($3$-dimensional) $\mathbb E$
on which $\Upsilon$ lives. Therefore, the study of commuting
operator triples that have $\overline{\mathbb E}$ as a spectral
set, takes us back to the complex geometry of $\overline{\mathbb
E}$.

\begin{defn}
A commuting triple of operators $(T_1,T_2,T_3)$ that has
$\overline{\mathbb E}$ as a spectral set is called an $\mathbb
E$-\textit{contraction}, i.e, an $\mathbb E$-contraction is a
commuting triple $(T_1,T_2,T_3)$ such that
$\sigma_T(T_1,T_2,T_3)\in \overline{\mathbb E}$ and that for every
holomorphic polynomial $p$ in three variables
\[
\| p(T_1,T_2,T_3) \|\leq \sup_{(x_1,x_2,x_3)\in\overline{\mathbb
E}}\, |p(x_1,x_2,x_3)|=\| p \|_{\infty, \overline{\mathbb E}}.
\]
\end{defn}
Since the set $\mathbb E$ is contained in the tridisc $\mathbb
D^3$ as was shown in \cite{awy}, an $\mathbb E$-contraction
consists of commuting contractions. Also it is merely mentioned
that if $(T_1,T_2,T_3)$ is an $\mathbb E$-contraction then so is
the adjoint $(T_1^*,T_2^*,T_3^*)$. In \cite{tirtha}, Bhattacharyya
introduced the study of $\mathbb E$-contractions by using the
efficient machinery so called \textit{fundamental operators}. We
mention here that in \cite{tirtha}, a triple $(T_1,T_2,T_3)$ that
had $\overline{\mathbb E}$ as a spectral set was called a
\textit{tetrablock contraction}. But since a notation is always
more convenient when writing, we prefer to call them $\mathbb
E$-contractions. It was shown in Theorem 3.5 in \cite{tirtha} that
to every $\mathbb E$-contraction $(T_1,T_2,T_3)$, there are two
unique operators $A_1,A_2$ on $\mathcal D_{T_3}$ such that
\[
T_1-T_2^*T_3=D_{T_3}A_1D_{T_3} \text{ and }
T_2-T_1^*T_3=D_{T_3}A_2D_{T_3}.
\]
For a contraction $T$, we shall always denote by $D_T$ the
positive operator $(I-T^*T)^{\frac{1}{2}}$ and $\mathcal
D_T=\overline{Ran}\,D_T$. An explicit dilation was constructed in
\cite{tirtha} for a particular class of $\mathbb E$-contractions
(see Theorem 6.1 in \cite{tirtha}). Since these two operators
$A_1,A_2$ were the key ingredients in that construction, they were
named the \textit{fundamental operators} of $(T_1,T_2,T_3)$. The
fundamental operators always satisfy $\omega(A_1+A_2z)\leq 1$, for
all $z$ in $\mathbb T$. For a further reading on $\mathbb
E$-contractions, fundamental operators and their properties, see
\cite{tirtha-sau, tirtha-sau-lata}. Also in a different direction
to know about operator theory and failure of rational dilation on
the tetrablock, an interested reader is
referred to \cite{sourav1, sourav2}.\\

Unitaries, isometries and co-isometries are special types of
contractions. There are natural analogues of these classes for
$\mathbb E$-contractions in the literature.
\begin{defn}
Let $T_1,T_2,T_3$ be commuting operators on a Hilbert space
$\mathcal H$. We say that $(T_1,T_2,T_3)$ is
\begin{itemize}
\item [(i)] an $\mathbb E$-\textit{unitary} if $T_1,T_2,T_3$ are
normal operators and the joint spectrum $\sigma_T(T_1,T_2,T_3)$ is
contained in $b\mathbb E$ ; \item [(ii)] an $\mathbb
E$-\textit{isometry} if there exists a Hilbert space $\mathcal K$
containing $\mathcal H$ and an $\mathbb E$-unitary
$(\tilde{T_1},\tilde{T_2},\tilde{T_3})$ on $\mathcal K$ such that
$\mathcal H$ is a common invariant subspace of $T_1,T_2,T_3$ and
that $T_i=\tilde{T_i}|_{\mathcal H}$ for $i=1,2,3$ ; \item [(iii)]
an $\mathbb E$-\textit{co-isometry} if $(T_1^*,T_2^*,T_3^*)$ is an
$\mathbb E$-isometry.
\end{itemize}
\end{defn}

\begin{defn}
An $\mathbb E$-contraction $(T_1,T_2,T_3)$ is said to be
\textit{pure} if $T_3$ is a pure contraction, i.e, ${T_3^*}^n
\rightarrow 0$ strongly as $n \rightarrow \infty$. Similarly an
$\mathbb E$-isometry $(T_1,T_2,T_3)$ is said to be a pure $\mathbb
E$-isometry if $T_3$ is a pure isometry, i.e, equivalent to a
shift operator.
\end{defn}
\begin{defn}
Let $(T_1,T_2,T_3)$ be a $\mathbb E$-contraction on $\mathcal H$.
A commuting triple $(Q_1,Q_2,V)$ on $\mathcal K$ is said to be an
$\mathbb E$-isometric dilation of $(T_1,T_2,T_3)$ if $\mathcal H
\subseteq \mathcal K$, $(Q_1,Q_2,V)$ is an $\mathbb E$-isometry
and
\[
P_{\mathcal H}(Q_1^{m_1}Q_2^{m_2}V^n)|_{\mathcal
H}=T_1^{m_1}T_2^{m_2}T_3^n, \; \textup{ for all non-negative
integers }m_1,m_2,n.
\]
Here $P_{\mathcal H}:\mathcal K \rightarrow
\mathcal H$ is the orthogonal projection of $\mathcal K$ onto
$\mathcal H$. Moreover, the dilation is called {\em minimal} if
\[
\mathcal K=\overline{\textup{span}}\{ Q_1^{m_1}Q_2^{m_2}V^n h\,:\;
h\in\mathcal H \textup{ and }m_1,m_2,n\in \mathbb N \cup \{0\} \}.
\]
\end{defn}

In section 3, we add to the account some operator theory on the
tetrablock. In Theorem \ref{dilation-theorem}, another main result
of this paper, we construct an $\mathbb E$-isometric dilation to a
pure $\mathbb E$-contraction $(T_1,T_2,T_3)$ whose adjoint has
commuting fundamental operators ${A_1}_*,{A_2}_*$ such that
$[{A_1}_*^*,{A_1}_*]=[{A_2}_*^*,{A_2}_*]$. This dilation is
different from the one established in \cite{tirtha} (Theorem 6.1
in \cite{tirtha}) and here the dilation operators involve the
fundamental operators of the adjoint $(T_1^*,T_2^*,T_3^*)$. We
show further that the dilation is minimal. As a consequence of
this dilation, we obtain a functional model in Theorem
\ref{modelthm} for such pure $\mathbb E$-contractions in terms of
commuting Toeplitz operators on the vectorial Hardy space
$H^2(\mathcal D_{T_3^*})$. Theorem \ref{sufficient1} describes a
set of sufficient conditions under which a triple of commuting
contractions $(T_1,T_2,T_3)$ becomes an $\mathbb E$-contraction.
Indeed, if there are two commuting operators $A_1,A_2$ that
satisfy $T_1-T_2^*T_3=D_{T_3}A_1D_{T_3}$ and
$T_2-T_1^*T_3=D_{T_3}A_2D_{T_3}$, then $(T_1,T_2,T_3)$ is an
$\mathbb E$-contraction provided that $[A_1^*,A_1]=[A_2^*,A_2]$
and $\omega(A_1+A_2z)\leq 1$, for every $z$ in the unit circle.
Also in Corollary \ref{dv11}, we show that a pair of commuting
operators $A_1,A_2$ on a Hilbert space $E$, satisfying
$[A_1^*,A_1]=[A_2^*,A_2]$ and $\omega(A_1+A_2z)\leq 1$, for all
$z$ of unit modulus, are the fundamental operators of an $\mathbb
E$-contraction defined on the vectorial Hardy space $H^2(E)$. This
can be treated as a partial converse to the existence-uniqueness
theorem (Theorem 3.5 in \cite{tirtha}) of fundamental operators.
Therefore, Theorem \ref{thm:DVchar} can be rephrased in the
following way: every distinguished variety in $\mathbb E$ can be
represented as
\[
\{ (x_1,x_2,x_3)\in\mathbb E\,:\, (x_1,x_2)\in
\sigma_T(A_1^*+x_3A_2, A_2^*+x_3A_1) \}
\]
where $A_1,A_2$ are the fundamental operators of an $\mathbb
E$-contraction. In Theorem \ref{char:1}, we characterize all
distinguished varieties for which $\|A_1^*+A_2z\|_{\infty, \mathbb
T}<1$.\\

We describe Theorem \ref{thm:VN} again from a different view
point; if the adjoint of an $\mathbb E$-contraction
$(T_1,T_2,T_3)$ is a pure $\mathbb E$-contraction and if the
fundamental operators $A_1,A_2$ of $(T_1,T_2,T_3)$ are commuting
matrices satisfying $[A_1^*,A_1]=[A_2^*,A_2]$, then Theorem
\ref{thm:VN} shows the existence of a one-dimensional subvariety
on which von-Neumann's inequality holds. Indeed, Theorem
\ref{dilation-theorem} provides an $\mathbb E$-co-isometric
extension of such a $(T_1,T_2,T_3)$
that naturally lives on that subvariety.\\

In section 5, we describe a connection between the distinguished
varieties in the tetrablock with that in the bidisc $\mathbb D^2$
and in the symmetrized bidisc $\mathbb G$, where
\[
\mathbb G=\{ (z_1+z_2,z_1z_2)\,:\, |z_1|<1,|z_2|<1 \}\subseteq
\mathbb C^2.
\]
Indeed, in Theorem \ref{connection}, we show that every
distinguished variety in $\mathbb E$ gives rise to a distinguished
variety in $\mathbb D^2$ as well as a distinguished variety in
$\mathbb G$.\\

In section 2, we briefly describe the Taylor joint spectrum of a
pair of commuting bounded operators and also recall from the
literature some results about the $\mathbb E$-contractions.\\

\noindent \textbf{Note.} After writing this paper, we learned that
Theorem \ref{modelthm} of this paper has been established
independently in \cite{sau} by Sau in non-commutative setting.

\section{Taylor joint spectrum and preliminary results about $\mathbb E$-contractions}

\subsection{Taylor joint spectrum} Here we briefly describe the Taylor joint spectrum of a
pair of commuting bounded operators and show how in case of
commuting matrices it becomes just the set of joint eigenvalues.
In fact all notions of joint spectrum (left/right Hart, Taylor)
are the same for a pair of commuting matrices. Let
$\underline{T}=(T_1,T_2)$ be a pair commuting bounded operators on
a Banach space $X$. Let us consider the complex

\begin{equation}\label{Koszul}
0\rightarrow X \xrightarrow{\delta_1} X\oplus X
\xrightarrow{\delta_2} X \rightarrow 0 \,,
\end{equation}
where $\delta_1$ and $\delta_2$ are defined by
$\delta_1x=T_1x\oplus T_2x$ $(x\in X)$ and $\delta_2(x_1\oplus
x_2)=-T_2x_1+T_1x_2$ $(x_1,x_2\in X)$. Clearly $T_1T_2=T_2T_1$
implies that $\delta_2\circ \delta_1=0$ so that (\ref{Koszul}) is
a chain complex. This chain complex is called the \textit{Koszul
complex} of $\underline{T}$. To say that the Koszul complex of
$\underline{T}$ is \textit{exact} means three things: at the first
and the third stage one has respectively
\[
Ker(T_1)\cap Ker(T_2)=\{0\} \text{ and } Ran(\delta_2)=X.
\]
In the second stage it means that $Ran{\delta_1}=Ker(\delta_2)$,
that is, for every point $x_1\oplus x_2 \in X\oplus X$ for which
$-T_2x_1+T_1x_2=0$ has the form $T_1x\oplus T_2x$ for some $x\in
X$. $(T_1,T_2)$ is said to be \textit{Taylor regular} if its
Koszul complex is exact. A point $\lambda=(\lambda_1,\lambda_2)\in
\mathbb C^2$ is said to belong to $\sigma_T(T_1,T_2)$, the
\textit{Taylor joint spectrum} of $(T_1,T_2)$, if $(T_1-\lambda_1,
T_2-\lambda_2)$ is not Taylor regular. For an explicit description
of Taylor joint spectrum in general setting, i.e, of an $n$-tuple
of commuting bounded operators one can see \cite{Taylor, Taylor1,
Muller}.\\

We now show that $\sigma_T(T_1,T_2)$ is just the set of joint
eigenvalues of $(T_1,T_2)$ when $T_1,T_2$ are matrices. We need
the following results before that.

\begin{lem}\label{spectra}
Let $X_1,X_2$ are Banach spaces and $A,D$ are bounded operators on
$X_1$ and $X_2$. Let $B\in\mathcal B(X_2,X_1)$. Then
\[
\sigma \left(\begin{bmatrix} A&B\\ 0&D \end{bmatrix}
\right)\subseteq \sigma(A)\cup \sigma(D)\,.
\]
\end{lem}
See Lemma 1 in \cite{hong} for a proof of this result.

\begin{lem}\label{spectra1}
Let $\underline{T}=(T_1,T_2)$ be a commuting pair of matrices on
an $N$-dimensional vector space $X$. Then there exists $N+1$
subspaces $L_0,L_1,\hdots,L_N$ satisfying:
\begin{enumerate}
\item $\{0\}=L_0\subseteq L_1 \subseteq \cdots \subseteq L_N=X$ ,
\item $L_k$ is $k$-dimensional for $k=1,\hdots,N$ , \item each
$L_k$ is a joint invariant subspace of $T_1, T_2$.
\end{enumerate}
\end{lem}
\begin{proof}
It is elementary to see that for a pair of commuting matrices
$(T_1,T_2)$ the set of joint eigenvalues is non-empty. Therefore,
there exists a vector $x_1\in X$ such that $x_1$ is a joint
eigenvector of $T_1$ and $T_2$. Let $L_1$ be the one-dimensional
subspace spanned by $x_1$. Then $L_1$ is invariant under
$T_1,T_2$. Next consider the vector space $Y=X/L_1$ and the linear
transformations $\tilde{T_1},\tilde{T_2}$ on $Y$ defined by
$\tilde{T_i}(x+L_1)=T_ix+L_1$. Then $\tilde{T_1},\tilde{T_2}$ are
commuting matrices and again they have a joint eigenvalue, say
$(\mu_1,\mu_2)$ and consequently a joint eigenvector, say
$x_2+L_1$. Thus $\tilde{T_i}(x_2+L_1)=\mu_ix_2+L_1$ for $i=1,2$
which means that $T_ix_2=\mu_ix_2+z$ for some $z\in L_1$. Hence
the subspace spanned by $x_1,x_2$ is invariant under $T_1,T_2$. We
call this subspace $L_2$ and it is two-dimensional with
$L_1\subseteq L_2$. Now applying the same reasoning to $X/L_2$ and
so on, we get for each $i=1,\hdots,N-1$ the subspace $L_i$ spanned
by $x_1,\hdots,x_i$. These subspaces satisfy the conditions of the
theorem. Finally, to complete the proof we define $L_N=X$.
\end{proof}

Let us choose an arbitrary $x_N\in L_N\setminus L_{N-1}$. Then
$\{x_1,\hdots,x_N\}$ is a basis for $X$ and with respect to this
basis the matrices $T_1,T_2$ are upper-triangular, i.e., of the
form
\[
\begin{pmatrix}
\lambda_1 & \ast & \ast & \ast \\
0&\lambda_2 & \ast & \ast \\
\vdots & \vdots &\ddots & \vdots \\
0 & 0 & \cdots & \lambda_N
\end{pmatrix}\,,\,
\begin{pmatrix}
\mu_1 & \ast & \ast & \ast \\
0&\mu_2 & \ast & \ast \\
\vdots & \vdots &\ddots & \vdots \\
0 & 0 & \cdots & \mu_N
\end{pmatrix}\,,
\]
where each $(\lambda_i,\mu_i)$ is called a \textit{joint diagonal
co-efficient} of $(T_1,T_2)$. Let us denote
$\sigma_{dc}(T_1,T_2)=\{(\lambda_i,\mu_i)\,:\, i=1,\cdots,N \}$.\\

The following result is well known and an interested reader can
see \cite{curto} for further details. We present here a simple and
straight forward proof to this.

\begin{thm}
Let $(T_1,T_2)$ be a pair of commuting matrices of order $N$ and
$\sigma_{pt}(T_1,T_2)$ be the set of joint eigenvalues of
$(T_1,T_2)$. Then
\[
\sigma_T(T_1,T_2)=\sigma_{pt}(T_1,T_2)=\sigma_{dc}(T_1,T_2).
\]
\end{thm}
\begin{proof}
We prove this Theorem by repeated application of Lemma
\ref{spectra} to the simultaneous upper-triangularization of Lemma
\ref{spectra1}. It is evident that each $(\lambda_i,\mu_i)$ is a
joint eigenvalue of $(T_1,T_2)$ and for each $(\lambda_i,\mu_i)$,
$Ker(T_1-\lambda_i)\cap Ker(T_2-\mu_i) \neq \emptyset$ which means
that the Koszul complex (see (\ref{Koszul})) of
$(T_1-\lambda_i,T_2-\mu_i)$ is not exact at the first stage and
consequently $(T_1-\lambda_i,T_2-\mu_i)$ is not Taylor-regular.
Therefore, $(\lambda_i,\mu_i)\in\sigma_T(T_1,T_2)$. Therefore,
\[
\sigma_{dc}(T_1,T_2)\subseteq \sigma_{pt}(T_1,T_2)\subseteq
\sigma_T(T_1,T_2).
\]
Now let $X_2$ be the subspace spanned by $x_2,\cdots,x_N$. Then
$X_2$ is $N-1$ dimensional and $X=L_1\oplus X_2$. For $i=1,2$ we
define $D_i$ on $X_2$ by the $(N-1)\times (N-1)$ matrices
\[
D_1=\begin{pmatrix}
\lambda_2 & \ast & \ast & \ast \\
0&\lambda_3 & \ast & \ast \\
\vdots & \vdots &\ddots & \vdots \\
0 & 0 & \cdots & \lambda_N
\end{pmatrix}\,,\,
D_2=\begin{pmatrix}
\mu_2 & \ast & \ast & \ast \\
0&\mu_3 & \ast & \ast \\
\vdots & \vdots &\ddots & \vdots \\
0 & 0 & \cdots & \mu_N
\end{pmatrix}\,.
\]
Then $(D_1,D_2)$ is a commuting pair as $(T_1,T_2)$ is so. Now we
apply Lemma \ref{spectra} and get $\sigma_T(T_1,T_2)\subseteq
\{(\lambda_1,\mu_1)\}\cup \sigma_T(D_1,D_2)$. Repeating this
argument $N$ times we obtain $\sigma_T(T_1,T_2)\subseteq
\sigma_{dc}(T_1,T_2)$. Hence we are done.
\end{proof}

\subsection{Preliminary results about $\mathbb E$-contractions}
By virtue of polynomial convexity of $\mathbb E$, the
condition on the Taylor joint spectrum can be avoided and the
definition of $\mathbb E$-contraction can be given only in terms
of von-Neumann's inequality as the following lemma shows.

\begin{lem} \label{simpler}
A commuting triple of bounded operators $(T_1,T_2,T_3)$ is an
$\mathbb E$-contraction if and only if $\| f (T_1,T_2,T_3) \| \leq
\| f \|_{\infty, \overline{\mathbb E}}$ for any holomorphic
polynomial $f$ in three variables.
\end{lem}
See Lemma 3.3 of \cite{tirtha} for a proof. Let us recall that the
{\em numerical radius} of an operator $T$ on a Hilbert space
$\mathcal{H}$ is defined by
\[
\omega(T) = \sup \{|\langle Tx,x \rangle|\; : \;
\|x\|_{\mathcal{H}}= 1\}.
\]
It is well known that
\begin{eqnarray}\label{nradius}
r(T)\leq \omega(T)\leq \|T\| \textup{ and } \frac{1}{2}\|T\|\leq
\omega(T)\leq \|T\|, \end{eqnarray} where $r(T)$ is the spectral
radius of $T$. We state two basic results about numerical radius
of which the first result has a routine proof. We shall use these
two results in sequel.

\begin{lem} \label{basicnrlemma}
The numerical radius of an operator $T$ is not greater than $1$ if
and only if  Re $\beta T \leq I$ for all complex numbers $\beta$
of modulus $1$.
\end{lem}

\begin{lem}\label{funda-properties}
Let $A_1,A_2$ be two bounded operators such that
$\omega(A_1+A_2z)\leq 1$ for all $z\in\mathbb T$. Then
$\omega(A_1+zA_2^*)\leq 1$ and $\omega(A_1^*+A_2z)\leq 1$ for all
$z\in\mathbb T$.

\end{lem}

\begin{proof}

We have that $\omega(A_1+zA_2)\leq 1$ for every $z\in\mathbb T$,
which is same as saying that $\omega(z_1A_1+z_2A_2)\leq 1$ for all
complex numbers $z_1,z_2$ of unit modulus. Thus by Lemma
\ref{basicnrlemma},
$$ (z_1A_1+z_2A_2)+(z_1A_1+z_2A_2)^*\leq 2I, $$
that is
$$ (z_1A_1+\bar{z_2}A_2^*)+(z_1A_1+\bar{z_2}A_2^*)^* \leq 2I.$$
Therefore, $z_1(A_1+zA_2^*)+\bar{z_1}(A_1+zA_2^*)^*\leq 2I$ for
all $z,z_1 \in\mathbb T$. This is same as saying that
$$ \textup{Re }z_1(A_1+zA_2^*)\leq I, \textup{ for all } z,z_1 \in\mathbb T. $$
Therefore, by Lemma \ref{basicnrlemma} again
$\omega(A_1+zA_2^*)\leq 1$ for any $z$ in $\mathbb T$. The proof
of $\omega(A_1^*+A_2z)\leq 1$ is similar.

\end{proof}

We recall from section 1, the existence-uniqueness theorem
(\cite{tirtha}, Theorem 3.5) for the fundamental operators of an
$\mathbb E$-contraction.

\begin{thm}\label{funda-exist-unique}
Let $(T_1,T_2,T_3)$ be an $\mathbb E$-contraction. Then there are
two unique operators $A_1,A_2$ in $\mathcal L(\mathcal D_{T_3})$
such that
\begin{equation}\label{basiceqn} T_1-T_2^*T_3=D_{T_3}A_1D_{T_3}\textup{ and }
T_2-T_1^*T_3=D_{T_3}A_2D_{T_3}. \end{equation} Moreover,
$\omega(A_1+zA_2)\leq 1$ for all $z\in\overline{\mathbb D}$.
\end{thm}
The following theorem gives a characterization of the set of
$\mathbb E$-unitaries (Theorem 5.4 in \cite{tirtha}).
\begin{thm}\label{tu}
Let $\underline N = (N_1, N_2, N_3)$ be a commuting triple of
bounded operators. Then the following are equivalent.

\begin{enumerate}

\item $\underline N$ is an $\mathbb E$-unitary,

\item $N_3$ is a unitary, $N_2$ is a contraction and $N_1 = N_2^*
N_3$,

\item $N_3$ is a unitary and $\underline N$ is an $\mathbb
E$-contraction.
\end{enumerate}
\end{thm}

\noindent Here is a structure theorem for the $\mathbb
E$-isometries and a proof can be found in \cite{tirtha} (see
Theorem 5.6 and Theorem 5.7 in \cite{tirtha}).
\begin{thm} \label{ti}

Let $\underline V = (V_1, V_2, V_3)$ be a commuting triple of
bounded operators. Then the following are equivalent.

\begin{enumerate}

\item $\underline V$ is an $\mathbb E$-isometry.

\item $\underline V$ is an $\mathbb E$-contraction and $V_3$ is an
    isometry.

\item $V_3$ is an isometry, $V_2$ is a contraction and $V_1=V_2^*
V_3$.

\item (\textit{Wold decomposition}) $\mathcal H$ has a
decomposition $\mathcal H=\mathcal H_1\oplus \mathcal H_2$ into
reducing subspaces of $V_1,V_2,V_3$ such that $(V_1|_{\mathcal
H_1},V_2|_{\mathcal H_1},V_3|_{\mathcal H_1})$ is an $\mathbb
E$-unitary and $(V_1|_{\mathcal H_2},V_2|_{\mathcal
H_2},V_3|_{\mathcal H_2})$ is a pure $\mathbb E$-isometry.

\end{enumerate}

\end{thm}

\section{Dilation and functional model for a subclass of pure $\mathbb E$-contractions}

\noindent We make a change of notation for an $\mathbb
E$-contraction in this section. Throughout this section we shall
denote an $\mathbb E$-contraction by $(T_1,T_2,T)$.

\begin{prop}\label{dilation-extension}
Let $(Q_1,Q_2,V)$ on $\mathcal K$ be an $\mathbb E$-isometric
dilation of an $\mathbb E$-contraction $(T_1,T_2,T)$ on $\mathcal
H$. If $(Q_1,Q_2,V)$ is minimal, then $(Q_1^*,Q_2^*,V^*)$ is an
$\mathbb E$-co-isometric extension of $(T_1^*,T_2^*,T^*)$.
Conversely, the adjoint of an $\mathbb E$-co-isometric extension
of $(T_1,T_2,T)$ is an $\mathbb E$-isometric dilation of
$(T_1,T_2,V)$.
\end{prop}
\begin{proof}
We first prove that $T_1P_{\mathcal H}=P_{\mathcal H}Q_1,
T_2P_{\mathcal H}=P_{\mathcal H}Q_2$ and $TP_{\mathcal
H}=P_{\mathcal H}V$. Clearly
$$\mathcal K=\overline{\textup{span}}\{ Q_1^{m_1}Q_2^{m_2}V^n h\,:\;
h\in\mathcal H \textup{ and }m_1,m_2,n\in \mathbb N \cup \{0\}
\}.$$ Now for $h\in\mathcal H$ we have that
\begin{align*} T_1P_{\mathcal
H}(Q_1^{m_1}Q_2^{m_2}V^n h) =T_1(T_1^{m_1}T_2^{m_2}T_3^n h)
&=T_1^{m_1+1}T_2^{m_2}T_3^n h \\& =P_{\mathcal
H}(Q_1^{m_1+1}Q_2^{m_2}V^n h)\\& =P_{\mathcal
H}Q_1(Q_1^{m_1}Q_2^{m_2}V^n h).
\end{align*}
Thus, $T_1P_{\mathcal H}=P_{\mathcal H}Q_1$. Similarly we can
prove that $T_2P_{\mathcal H}=$ $P_{\mathcal H}Q_2$ and that
$T_3P_{\mathcal H}=P_{\mathcal H}V$. Also for $h\in\mathcal H$ and
$k\in\mathcal K$ we have that
\begin{align*}
\langle T_1^*h,k \rangle =\langle P_{\mathcal H}T_1^*h,k \rangle
=\langle T_1^*h,P_{\mathcal H}k \rangle  =\langle h,T_1P_{\mathcal
H}k \rangle &=\langle h,P_{\mathcal H}Q_1k \rangle \\&=\langle
Q_1^*h,k \rangle .
\end{align*}
Hence $T_1^*=Q_1^*|_{\mathcal H}$ and similarly
$T_2^*=Q_2^*|_{\mathcal H}$ and $T^*=V^*|_{\mathcal H}$.
Therefore, $(Q_1^*,Q_2^*,V^*)$ is an $\mathbb E$-co-isometric
extension of $(T_1^*,T_2^*,T^*)$.

The converse part is obvious.

\end{proof}

Let us recall from \cite{nagy}, the notion of the characteristic
function of a contraction $T$. For a contraction $T$ defined on a
Hilbert space $\mathcal H$, let $\Lambda_T$ be the set of all
complex numbers for which the operator $I-zT^*$ is invertible. For
$z\in \Lambda_T$, the characteristic function of $T$ is defined as
\begin{eqnarray}\label{e0} \Theta_T(z)=[-T+zD_{T^*}(I-zT^*)^{-1}D_T]|_{\mathcal D_T}.
\end{eqnarray} Here the operators $D_T$ and $D_{T^*}$ are the
defect operators $(I-T^*T)^{1/2}$ and $(I-TT^*)^{1/2}$
respectively. By virtue of the relation $TD_T=D_{T^*}P$ (section
I.3 of \cite{nagy}), $\Theta_T(z)$ maps $\mathcal
D_T=\overline{\textup{Ran}}D_T$ into $\mathcal
D_{T^*}=\overline{\textup{Ran}}D_{T^*}$ for every $z$ in
$\Lambda_T$.

Let us recall that a pure contraction $T$ is a contraction such
that ${T^*}^n \rightarrow 0$ strongly as $n \rightarrow \infty$.
It was shown in \cite{nagy} that every pure contraction $T$
defined on a Hilbert space $\mathcal H$ is unitarily equivalent to
the operator $\mathbb T=P_{\mathbb H_T}(M_z\otimes I)|_{\mathbb
H_T}$ on the Hilbert space $\mathbb H_T=(H^2(\mathbb D)\otimes
\mathcal D_{T^*}) \ominus M_{\Theta_T}(H^2(\mathbb D)\otimes
\mathcal D_T)$, where $M_z$ is the multiplication operator on
$H^2(\mathbb D)$ and $M_{\Theta_T}$ is the multiplication operator
from $H^2(\mathbb D)\otimes \mathcal D_T$ into $H^2(\mathbb
D)\otimes \mathcal D_{T^*}$ corresponding to the multiplier
$\Theta_T$. Here, in an analogous way, we produce a model for a
subclass of pure $\mathbb E$-contractions $(T_1,T_2,T)$.

\begin{thm} \label{dilation-theorem}
Let $(T_1,T_2,T)$ be a pure $\mathbb E$-contraction on a Hilbert
space $\mathcal{H}$ and let the fundamental operators
$A_{1*},A_{2*}$ of $(T_1^*,T_2^*,T^*)$ be commuting operators
satisfying $[A_{1*}^*,A_{1*}]=[A_{2*}^*,A_{2*}]$ Consider the
operators $Q_1,Q_2,V$ on $\mathcal{K}=H^2(\mathbb{D}) \otimes
\mathcal{D}_{T^*}$ defined by
\[
Q_1=I\otimes A_{1*}^*+M_z\otimes A_{2*},\quad Q_2=I\otimes
A_{2*}^*+M_z\otimes A_{1*} \text{ and } V=M_z\otimes I.
\]
Then $(Q_1,Q_2,V)$ is a minimal pure $\mathbb E$-isometric
dilation of $(T_1,T_2,T)$.
\end{thm}
\begin{proof}
The minimality is obvious if we prove that $(Q_1,Q_2,V)$ is an
$\mathbb E$-isometric dilation of $(T_1,T_2,T)$. This is because
$V$ on $\mathcal K$ is the minimal isometric dilation for $T$.
Therefore, by virtue of Lemma \ref{dilation-extension}, it
suffices if we show that $(Q_1^*,Q_2^*,V^*)$ is an $\mathbb
E$-co-isometric extension of $(T_1^*,T_2^*,T^*)$. Since the
fundamental operators $A_{1*},A_{2*}$ commute and satisfy
$[A_{1*}^*,A_{1*}]=[A_{2*}^*,A_{2*}]$, $Q_1,Q_2$ and $V$ commute.
Also it is evident that $V$ is a pure isometry. Thus invoking
Theorem \ref{ti}, we need to verify the following steps.
\begin{enumerate}
\item $Q_1=Q_2^*V, \; \| Q_2 \|\leq 1$. \item There is an isometry
$W:\mathcal H \rightarrow H^2 \otimes {\mathcal D}_{T^*}$ such
that\\ $ Q_1^*|_{W(\mathcal H)}=WT_1^*W^*|_{W(\mathcal H)},\,
Q_2^*|_{W(\mathcal H)}=WT_2^*W^*|_{W(\mathcal H)}$\\ and $
V^*|_{W(\mathcal H)}=WT^*W^*|_{W(\mathcal H)}$.
\end{enumerate}

\noindent \textbf{Step 1.} $Q_1=Q_2^*V$ is obvious. $\| Q_2 \|\leq
1$ follows from Lemma \ref{funda-properties}.\\

\noindent \textbf{Step 2.} Let us define $W$ by
\begin{align*}
W:& \mathcal{H} \rightarrow \mathcal{K} \\& h\mapsto
\sum_{n=0}^{\infty} z^n\otimes D_{T^*}{T^*}^n h.
\end{align*}

Now
\begin{align*}
\|Wh\|^2 &= \|\displaystyle
\sum_{n=0}^{\infty}{z^n\otimes D_{T^*}{T^*}^n h}\|^2 \\&= \langle
\displaystyle \sum_{n=0}^{\infty}{z^n\otimes D_{T^*}{T^*}^n
h}\;,\;\displaystyle \sum_{m=0}^{\infty}{z^m\otimes D_{T^*}{T^*}^m
h} \rangle
\\& = \displaystyle \sum_{m,n=0}^{\infty} \langle z^n,z^m \rangle
\langle D_{T^*}{T^*}^nh\;,\;D_{T^*}{T^*}^mh \rangle \\&
=\displaystyle \sum_{n=1}^{\infty}{\langle T^n D_{T^*}^2
{T^*}^nh,h \rangle}\\&= \displaystyle \sum_{n=0}^{\infty}\langle
T^n(I-TT^*){T^*}^nh,h \rangle\\& = \displaystyle
\sum_{n=0}^{\infty}\{\langle T^n{T^*}^nh,h \rangle-\langle
T^{n+1}{T^*}^{n+1}h,h \rangle\} \\&= \|h\|^2-\lim_{n \rightarrow
\infty}\|{T^*}^nh\|^2.
\end{align*}
Since $T$ is a pure contraction, $ \displaystyle \lim_{n\rightarrow
\infty}\|{T^*}^nh\|^2=0$ and hence $\|Wh\|=\|h\|.$ Therefore $W$
is an isometry.

For a basis vector $z^n\otimes \xi$ of $\mathcal{K}$ we have that
\begin{align*}
\langle W^*(z^n\otimes \xi),h \rangle  = \langle z^n \otimes \xi ,
\displaystyle \sum_{k=0}^{\infty}{z^k \otimes D_{T^*}{T^*}^kh}
\rangle  &= \langle \xi , D_{T^*}{T^*}^nh \rangle \\& = \langle
T^n D_{T^*}\xi , h \rangle.
\end{align*}
Therefore,
\begin{equation}\label{001}
W^*(z^n\otimes \xi)=T^n D_{T^*} \xi, \quad \text{ for }
n=0,1,2,3,\hdots
\end{equation}
and hence
\[
TW^*(z^n \otimes \xi)=T^{n+1} D_{T^*} \xi, \text{ for }
n=0,1,2,3,\hdots.
\]
Again by (\ref{001}),
\begin{align*}
W^*V(z^n \otimes \xi)=W^*(M_z \otimes I)(z^n \otimes \xi) =
W^*(z^{n+1} \otimes \xi) &= T^{n+1}D_{T^*}\xi \\& =TW^*(z^n\otimes
\xi).
\end{align*}
Consequently, $W^*V = TW^*$, i.e, $V^*W=WT^*$ and
hence $V^*|_{W(\mathcal H)}=WT^*W^*|_{W(\mathcal H)}$.\\

We now show that $W^*Q_1=T_1W^*$.
\begin{align*}
W^*Q_1(z^n \otimes \xi)&=W^*(I\otimes A_{1*}^*+M_z\otimes
A_{2*})(z^n \otimes \xi) \\&=W^*(z^n\otimes A_{1*}^*
\xi)+W^*(z^{n+1}\otimes A_{2*} \xi)\\&=T^nD_{T^*}A_{1*}^* \xi
+T^{n+1}D_{T^*}A_{2*} \xi .
\end{align*}
Also \begin{equation}\label{002} T_1W^*(z^n\otimes
\xi)=T_1T^nD_{T^*} \xi=T^nT_1D_{T^*} \xi . \end{equation}
\noindent \textit{Claim.} $T_1D_{T^*}=D_{T^*}A_{1*}^*+TD_{T^*}A_{2*}$.\\
\noindent \textit{Proof of Claim.} Since $A_{1*}, A_{2*}$ are the
fundamental operators of $(T_1^*,T_2^*,T^*)$, we have
\[
(D_{T^*}A_{1*}^*+TD_{T^*}A_{2*})D_{T^*}=(T_1-TT_2^*)+T(T_2^*-T_1T^*)=T_1D_{T^*}^2.
\]
Now if $G=T_1D_{T^*}-D_{T^*}A_{1*}^*-TD_{T^*}A_{2*}$, then $G$ is
defined from $\mathcal D_{T^*}$ to $\mathcal H$ and $GD_{T^*}h=0$
for every $h\in \mathcal D_{T^*}$. Hence the claim follows.\\

So from (\ref{002}) we have,
\[
T_1W^*(z^n\otimes \xi)=T^n(D_{T^*}A_{1*}^*+TD_{T^*}A_{2*}).
\]
Therefore, $W^*Q_1=T_1W^*$ and hence $Q_1^*|_{W(\mathcal
H)}=WT_1^*W^*|_{W(\mathcal H)}$. Similarly we can show that
$Q_2^*|_{W(\mathcal H)}=WT_2^*W^*|_{W(\mathcal H)}$. The proof is
now complete.

\end{proof}

In \cite{tirtha}, an explicit $\mathbb E$-isometric dilation was
constructed for an $\mathbb E$-contraction $(T_1,T_2,T_3)$ whose
fundamental operators satisfy $[A_1,A_2]=0$ and
$[A_1^*,A_1]=[A_2^*,A_2]$ (Theorem 6.1 in \cite{tirtha}). The
fundamental operators of $(T_1,T_2,T_3)$ were the key ingredients
in that construction. Such an explicit $\mathbb E$-isometric
dilation of an $\mathbb E$-contraction could be treated as an
analogue of Schaeffer's construction of isometric dilation of a
contraction. The dilation we provided in the previous theorem was
only to a pure $\mathbb E$-contraction and was different in the
sense that the fundamental operators of the adjoint
$(T_1^*,T_2^*,T_3^*)$ played the main role there. As a consequence
of the dilation theorem in \cite{tirtha}, we have the following
result.

\begin{thm}\label{sufficient1}
Let $T_1,T_2,T_3$ be commuting contractions on a Hilbert space
$\mathcal H$. Let $A_1,A_2$ be two commuting bounded operators on
$\mathcal D_{T_3}$ such that
\[
T_1-T_2^*T_3=D_{T_3}A_1D_{T_3} \text{ and }
T_2-T_1^*T_3=D_{T_3}A_2D_{T_3}.
\]
If $A_1,A_2$ satisfy $[A_1^*,A_1]=[A_2^*,A_2]$ and
$\omega(A_1+A_2z)\leq 1$, for all $z$ from the unit circle, then
$(T_1,T_2,T_3)$ is an $\mathbb E$-contraction.
\end{thm}
\begin{proof}
It is evident from Theorem 6.1 of \cite{tirtha} that such a triple
$(T_1,T_2,T_3)$ has an $\mathbb E$-isometric dilation and hence an
$\mathbb E$-unitary dilation. Therefore, $\overline{\mathbb E}$ is
a complete spectral set for $(T_1,T_2,T_3)$ and hence
$(T_1,T_2,T_3)$ is an $\mathbb E$-contraction.
\end{proof}

\begin{thm}\label{modelthm}
Let $(T_1,T_2,T)$ be a pure $\mathbb E$-contraction on a Hilbert
space $\mathcal{H}$ and let the fundamental operators
$A_{1*},A_{2*}$ of $(T_1^*,T_2^*,T^*)$ be commuting operators
satisfying $[A_{1*}^*,A_{1*}]=[A_{2*}^*, A_{2*}]$. Then
$(T_1,T_2,T)$ is unitarily equivalent to the triple $(R_1,R_2,R)$
on the Hilbert space $\mathbb H_T=(H^2(\mathbb D)\otimes \mathcal
D_{T^*})\ominus M_{\Theta_T}(H^2(\mathbb D)\otimes \mathcal D_T)$
defined by
\begin{align*}
 & R_1=P_{\mathbb H_T}(I\otimes
A_{1*}^*+M_z\otimes A_{2*})|_{\mathbb H_T},\; R_2=P_{\mathbb
H_T}(I\otimes A_{2*}^*+M_z\otimes A_{1*})|_{\mathbb H_T}\\&
\text{and } R=P_{\mathbb H_T}(M_z\otimes I)|_{\mathbb H_T}.
\end{align*}
\end{thm}

\begin{proof}
It suffices if we show that $W(\mathcal H)=\mathbb H_T$. For this,
it is enough if we can prove
\[
WW^*+M_{\Theta_T}M_{\Theta_T}^*=I_{H^2(\mathbb D)\otimes \mathcal
 D_{T^*}}.
\]
Since the vectors $z^n\otimes \xi$ forms a basis for $H^2(\mathbb
D)\otimes \mathcal D_{T^*},$ it is obvious from equation
(\ref{002}) that
\[
W^*(f\otimes \xi)=f(P)D_{P^*}\xi, \text{ for all } f \in \mathbb
C[z], \text{ and } \xi \in \mathcal D_{P^*}.
\]
It was shown in the proof of Theorem 1.2 of \cite{arveson3} by
Arveson that the operator $W^*$ satisfies the identity
\[
W^*(k_z\otimes \xi)=
 (I-\bar z T)^{-1}D_{T^*}\xi  \text{ for } z\in \mathbb D, \xi \in \mathcal
 D_{T^*},
 \]
 where $k_z(w)=(1-\langle w,z \rangle)^{-1}$.
 Therefore, for $z,w$ in $\mathbb{D}$ and $\xi,\eta$ in $\mathcal
 D_{T^*}$, we obtain
 \begin{align*}
 &\quad \langle (WW^*+M_{\Theta_T}M_{\Theta_T}^*)k_z\otimes \xi, k_w\otimes \eta
 \rangle\\&
 =\langle W^*(k_z\otimes \xi),W^*(k_w\otimes \eta) \rangle +
 \langle M_{\Theta_T}^*(k_z\otimes \xi), M_{\Theta_T}^*(k_w\otimes \eta)
 \rangle\\&
 =\langle(I-\bar{z}T)^{-1}D_{T^*}\xi,(I-\bar{w}T)^{-1}D_{T^*}\eta \rangle+
 \langle k_z\otimes \Theta_T(z)^*\xi,k_w\otimes \Theta_T(w)^*\eta
 \rangle\\&
 =\langle D_{T^*}(I-wT^*)^{-1}(I-\bar{z}T)^{-1}D_T^*\xi,\eta \rangle+
 \langle k_z,k_w \rangle\langle \Theta_T(w)\Theta_T(z)^*\xi,\eta
 \rangle\\&
 =\langle k_z\otimes \xi, k_w\otimes \eta \rangle.
 \end{align*}
 The last equality follows from the following identity (see page 244 in \cite{nagy}),
 \[
 1-\Theta_T(w)\Theta_T(z)^*=(1-w\bar{z})D_{T^*}(1-wT^*)^{-1}(1-\bar{z}T)^{-1}D_{T^*},
 \]
 where $\Theta_T$ is the characteristic function of $T$.
 Using the fact that the vectors $k_z$ form a total
 set in $H^2(\mathbb{D})$, the assertion follows.

\end{proof}

\begin{rem}
It is interesting to notice that the model space $\mathbb H_T$ and
model operator $R$ are same as the model space and model operator
of the pure contraction $T$ described in \cite{nagy}.
\end{rem}

The following theorem, which appeared in \cite{sourav1}, gives an
explicit model for pure $\mathbb E$-isometries.
\begin{thm}\label{model1}
Let $(\hat{T_1},\hat{T_2},\hat{T_3})$ be a commuting triple of
operators on a Hilbert space $\mathcal H$. If
$(\hat{T_1},\hat{T_2},\hat{T_3})$ is a pure $\mathbb E$-isometry
then there is a unitary operator $U:\mathcal H \rightarrow
H^2(\mathcal D_{{\hat{T_3}}^*})$ such that
$$ \hat{T_1}=U^*T_{\varphi}U,\quad \hat{T_2}=U^*T_{\psi}U \textup{ and }
\hat{T_3}=U^*T_zU,$$ where $\varphi(z)= A_1^*+A_2z,\,\psi(z)=
A_2^*+A_1z, \quad z\in\mathbb D$ and $A_1,A_2$ are the fundamental
operators of $(\hat{T_1}^*,\hat{T_2}^*,\hat{T_3}^*)$ such that
\begin{enumerate}
\item $[A_1,A_2]=0$ and $[A_1^*,A_1]=[A_2^*,A_2]$ \item
$\|A_1^*+A_2z\|_{\infty, \overline{\mathbb D}}\leq 1$.
\end{enumerate}
Conversely, if $A_1$ and $A_2$ are two bounded operators on a
Hilbert space $E$ satisfying the above two conditions, then
$(T_{A_1^*+A_2z},T_{A_2^*+A_1z},T_z)$ on $H^2(E)$ is a pure
$\mathbb E$-isometry.
\end{thm}
See Theorem 3.3 in \cite{sourav1} for a proof.

\begin{cor}\label{dv11}
Let $A_1,A_2$ be two commuting operators on a Hilbert space $E$
which satisfy $[A_1^*,A_1]=[A_2^*,A_2]$ and $\omega(A_1+A_2z)\leq
1$, for all $z\in\mathbb T$, Then $A_1,A_2$ are the fundamental
operator of an $\mathbb E$-contraction on $H^2(E).$
\end{cor}
\begin{proof}
By Lemma \ref{funda-properties}, $\omega(A_1^*+A_2z)\leq 1$ and
since $A_1^*+A_2z$ is normal, we have that
$\omega(A_1^*+A_2z)=\|A_1^*+A_2z\|\leq 1$, for all $z\in\mathbb
T$. Now it is clear from the previous theorem that $A_1,A_2$ are
the fundamental operators of
$(T_{A_1^*+A_2z}^*,T_{A_2^*+A_1z}^*,T_z^*)$ on $H^2(E)$.
\end{proof}

\section{Representation of a distinguished variety in the tetrablock}

\noindent We follow here the notations and terminologies used by
Agler and M$^{\textup{c}}$Carthy in \cite{AM05}. We say that a
function $f$ is \textit{holomorphic} on a distinguished variety
$\Omega$ in $\mathbb E$, if for every point of $\Omega$, there is
an open ball $B$ in $\mathbb C^3$ containing the point and a
holomorphic function $F$ of three variables on $B$ such that
$F|_{B\cap \Omega}=f|_{B \cap \Omega}$. We shall denote by
$A(\Omega)$ the Banach algebra of functions that are holomorphic
on $\Omega$ and continuous on $\overline{\Omega}$. This is a
closed unital subalgebra of $C(\partial \Omega)$ that separates
points. The maximal ideal space of $A(\Omega)$ is
$\overline{\Omega}$.

For a finite measure $\mu$ on $\Omega$, let $H^2(\mu)$ be the
closure of polynomials in $L^2(\partial \Omega, \mu)$. If $G$ is
an open subset of a Riemann surface $S$ and $\nu$ is a finite
measure on $\overline G$, let $\mathcal A^2(\nu)$ denote the
closure in $L^2(\partial G, \nu)$ of $A(G)$. A point $\lambda$ is
said to be a \textit{bounded point evaluation} for $H^2(\mu)$ or
$\mathcal A^2(\nu)$ if evaluation at $\lambda$, \textit{a priori}
defined on a dense set of analytic functions, extends continuously
to the whole Hilbert space $H^2(\mu)$ or $\mathcal A^2(\nu)$
respectively. If $\lambda$ is a bounded point evaluation, then the
function defined by
$$ f(\lambda)=\langle f,k_{\lambda} \rangle $$ is called the
\textit{evaluation functional at} $\lambda$. The following result
is due to Agler and M$^{\textup{c}}$Carthy (see Lemma 1.1 in
\cite{AM05}).
\begin{lem}\label{basiclem1}
Let $S$ be a compact Riemann surface. Let $G\subseteq S$ be a
domain whose boundary is a finite union of piecewise smooth Jordan
curves. Then there exists a measure $\nu$ on $\partial G$ such
that every point $\lambda$ in $G$ is a bounded point evaluation
for $\mathcal A^2(\nu)$ and such that the linear span of the
evaluation functional is dense in $\mathcal A^2(\nu)$.
\end{lem}

\begin{lem}\label{basiclem2}
Let $\Omega$ be a one-dimensional distinguished variety in
$\mathbb E$. Then there exists a measure $\mu$ on $\partial
\Omega$ such that every point in $\Omega$ is a bounded point
evaluation for $H^2(\mu)$ and such that the span of the bounded
evaluation functionals is dense in $H^2(\mu)$.
\end{lem}
\begin{proof}
Agler and M$^{\textup{c}}$Carthy proved a similar result for
distinguished varieties in the bidisc (see Lemma 1.2 in
\cite{AM05}); we imitate their proof here for the tetrablock.

Let $p,q$ be minimal polynomials such that
\[
\Omega=\{(x_1,x_2,x_3)\in \mathbb E\,:\,
p(x_1,x_2,x_3)=q(x_1,x_2,x_3)=0\}.
\]
Let $\mathbb Z_{pq}$ be the intersection of the zero sets of $p$
and $q$, i.e, $\mathbb Z_{pq}=\mathbb Z_p \cap \mathbb Z_q$. Let
$C(\mathbb Z_{pq})$ be the closure of $\mathbb Z_{pq}$ in the
projective space $\mathbb{CP}^3$. Let $S$ be the desingularization
of $C(\mathbb Z_{pq})$. See, e.g., \cite{fischer}, \cite{harris}
and \cite{griffiths} for details of desingularization. Therefore,
$S$ is a compact Riemann surface and there is a holomorphic map
$\tau: S \rightarrow C(\mathbb Z_{pq})$ that is biholomorphic from
$S^{\prime}$ onto $C(\mathbb Z_{pq})^{\prime}$ and finite-to-one
from $S\setminus S^{\prime}$ onto $C(\mathbb Z_{pq})\setminus
C(\mathbb Z_{pq})^{\prime}$. Here $C(\mathbb Z_{pq})^{\prime}$ is
the set of non-singular points in $C(\mathbb Z_{pq})$ and
$S^{\prime}$ is the pre-image of $C(\mathbb Z_{pq})^{\prime}$
under $\tau$.

Let $G=\tau^{-1}(\Omega)$. Then $\partial G$ is a finite union of
disjoint curves, each of which is analytic except possibly at a
finite number of cusps and $G$ satisfies the conditions of Lemma
\ref{basiclem1}. So there exists a measure $\nu$ on $\partial G$
such that every point in $G$ is a bounded point evaluation for
$\mathcal A^2(\nu)$. Let us define our desired measure $\mu$ by
\[
\mu(E)=\nu(\tau^{-1}(E)), \text{ for a Borel subset } E \text{ of
} \partial \Omega.
\]
Clearly, if $\lambda$ is in $G$ and $\tau(\eta)=\lambda$, let
$k_{\eta}\nu$ be a representing measure for $\eta$ in $A(G)$. Then
the function $k_{\eta}\circ \tau^{-1}$ is defined $\mu$-almost
everywhere and satisfies
\begin{gather*}
\int_{\partial \Omega}p(k_{\eta}\circ \tau^{-1})d\mu=
\int_{\partial G}(p\circ \tau)k_{\eta}d\nu =p\circ \tau
(\eta)=p(\lambda) \text{ and}\\
\int_{\partial \Omega}q(k_{\eta}\circ \tau^{-1})d\mu=
\int_{\partial G}(q\circ \tau)k_{\eta}d\nu =q\circ \tau
(\eta)=q(\lambda).
\end{gather*}

\end{proof}

\begin{lem}\label{lemeval}
Let $\Omega$ be a one-dimensional distinguished variety in
$\mathbb E$, and let $\mu$ be the measure on $\partial \Omega$
given as in Lemma \textup{\ref{basiclem2}}. A point $(y_1,y_2,y_3)
\in \mathbb E$ is in $\Omega$ if and only if $(\bar y_1, \bar
y_2,\bar y_3)$ is a joint eigenvalue for $M_{x_1}^*, M_{x_2}^*$
and $M_{x_3}^*$.
\end{lem}
\begin{proof}
It is a well known fact in the theory of reproducing kernel
Hilbert spaces that $M_f^* k_x = \overline{f(x)} k_x$ for every
multiplier $f$ and every kernel function $k_x$; in particular
every point $(\bar y_1, \bar y_2,\bar y_3) \in \Omega$ is a joint
eigenvalue for $M_{y_1}^*, M_{y_2}^*$ and $M_{y_3}^*$.

Conversely, if $(\bar y_1, \bar y_2,\bar y_3)$ is a joint
eigenvalue and $v$ is a unit eigenvector, then $f(y_1,y_2,y_3) =
\langle  v, M_f^* v\rangle$ for every polynomial $f$. Therefore,
\[
|f(y_1,y_2,y_3)| \leq \|M_f\| = \sup_{(x_1,x_2,x_3) \in \Omega}
|f(x_1,x_2,x_3)|.
\]
So $(y_1,y_2,y_3)$ is in the polynomial convex
hull of $\Omega$ (relative to $\mathbb E$), which is $\Omega$.
\end{proof}

\begin{lem}\label{lempure}
Let $\Omega$ be a one-dimensional distinguished variety in
$\mathbb E$, and let $\mu$ be the measure on $\partial \Omega$
given as in Lemma \textup{\ref{basiclem2}}. The multiplication
operator triple $(M_{x_1},M_{x_2},M_{x_3})$ on $H^2(\mu)$, defined
as multiplication by the co-ordinate functions, is a pure $\mathbb
E$-isometry.
\end{lem}
\begin{proof}
Let us consider the pair of operators
$(\widehat{M_{x_1}},\widehat{M_{x_2}},\widehat{M_{x_3}})$,
multiplication by co-ordinate functions, on $L^2(\partial \Omega,
\mu)$. They are commuting normal operators and the joint spectrum
$\mu(\widehat{M_{x_1}},\widehat{M_{x_2}},\widehat{M_{x_3}})$ is
contained in $\partial \Omega \subseteq b\mathbb E$. Therefore,
$(\widehat{M_{x_1}},\widehat{M_{x_2}},\widehat{M_{x_3}})$ is an
$\mathbb E$-unitary and $(M_{x_1},M_{x_2},M_{x_3})$, being the
restriction of
$(\widehat{M_{x_1}},\widehat{M_{x_2}},\widehat{M_{x_3}})$ to the
common invariant subspace $H^2(\mu)$, is an $\mathbb E$-isometry.
By a standard computation, for every $\overline y=(y_1,y_2,y_3)
\in \Omega$, the kernel function $k_{\bar y}$ is an eigenfunction
of $M_{x_3}^*$ corresponding to the eigenvalue $\overline{y_3}$.
Therefore,
\[
(M_{x_3}^*)^nk_{\overline y}=\overline{ y_3}^nk_{\overline y}
\rightarrow 0 \; \textup{ as } n \rightarrow \infty,
\]
because $|y_3|<1$ by Theorem \ref{thm1}. Since the evaluation
functionals $k_{\overline y}$ are dense in $H^2(\mu)$, this shows
that $M_{x_3}$ is pure. Hence $M_{x_3}$ is a pure isometry and
consequently $(M_{x_1},M_{x_2},M_{x_3})$ is a pure $\mathbb
E$-isometry on $H^2(\mu)$.
\end{proof}

We now present the main result of this section, the theorem that
gives a representation of a distinguished variety in $\mathbb E$
in terms of the natural coordinates of $\mathbb E$.

\begin{thm}\label{thm:DVchar}
Let
\begin{equation}\label{eq:W} \Omega = \{(x_1,x_2,x_3) \in
\mathbb E \,:\, (x_1,x_2)\in \sigma_T(A_1^*+x_3A_2,
A_2^*+x_3A_1)\},
\end{equation}
where $A_1,A_2$ are commuting square matrices of same order such
that
\begin{enumerate}
\item $[A_1^*,A_1]=[A_2^*,A_2]$ \item $\|A_1^*+A_2z\|_{\infty,
\mathbb T}<1$.
\end{enumerate}
Then $\Omega$ is a one-dimensional distinguished variety in
$\mathbb E$. Conversely, every distinguished variety in $\mathbb
E$ is one-dimensional and can be represented as (\ref{eq:W}) for
two commuting square matrices $A_1,A_2$ of same order, such that
\begin{enumerate}
\item $[A_1^*,A_1]=[A_2^*,A_2]$ \item $\|A_1^*+A_2z\|_{\infty,
\mathbb T}\leq 1$.
\end{enumerate}
\end{thm}

\begin{proof}
Suppose that
\[
\Omega = \{(x_1,x_2,x_3) \in \mathbb E \,:\, (x_1,x_2)\in
\sigma_T(A_1^*+x_3A_2, A_2^*+x_3A_1)\},
\]
where $A_1,A_2$ are commuting matrices of order $n$ satisfying the
given conditions. Then for any $x_3$, $A_1^*+x_3A_2$ and
$A_2^*+x_3A_1$ commute and consequently
$\sigma_T(A_1^*+x_3A_2,A_2^*+x_3A_1)\neq \emptyset$. We now show
that if $|x_3|<1$ and $(x_1,x_2)\in\sigma_T(A_1^*+x_3A_2,
A_2^*+x_3A_1)$ then $(x_1,x_2,x_3)\in \mathbb E$ which will
establish that $\Omega$ is non-empty and that it exits through the
distinguished boundary $b\mathbb E$. This is because proving the
fact that $\Omega$ exits through $b\mathbb E$ is same as proving
that $\overline{\Omega}\cap (\partial E \setminus bE)=\emptyset$,
i.e, if $(x_1,x_2,x_3)\in\overline{\Omega}$ and $|x_3|<1$ then
$(x_1,x_2,x_3)\in \mathbb E$ (by Theorem \ref{thm1}). Let
$|x_3|<1$ and $(x_1,x_2)$ be a joint eigenvalue of $A_1^*+x_3A_2$
and $A_2^*+x_3A_1$. Then there exists a unit joint eigenvector
$\nu$ such that $ (A_1^*+x_3A_2)\nu=x_1\nu \textup{ and }
(A_2^*+x_3A_1)\nu=x_2\nu $. Taking inner product with respect to
$\nu$ we get
\[
\alpha_1+\bar{\beta_1}x_3=x_1 \text{ and }
\beta_1+\bar{\alpha_1}x_3=x_2 \,,
\]
where $ \alpha_1=\langle A_1^*\nu,\nu \rangle $ and $
\beta_1=\langle A_2^*\nu,\nu \rangle$. Here $\alpha_1$ and
$\beta_1$ are unique because we have that $
x_1-\bar{x_2}x_3=\alpha_1(1-|x_3|^2)$ and $
x_2-\bar{x_1}x_3=\beta_1(1-|x_3|^2) $ which lead to
\[
\alpha_1=\frac{x_1-\bar{x_2}x_3}{1-|x_3|^2} \text{ and }
\beta_1=\frac{x_2-\bar{x_1}x_3}{1-|x_3|^2}.
\]
Since $A_1,A_2$ commute and $[A_1^*,A_1]=[A_2^*,A_2]$,
$A_1^*+A_2z$ is a normal matrix for every $z$ of unit modulus. So
we have that $\|A_1^*+A_2z\|=\omega(A_1^*+A_2z)<1$ and by Lemma
\ref{funda-properties}, $\omega(A_1+A_2z)<1$ for every $z$ in
$\mathbb T$. This implies that $\omega(z_1A_1^*+z_2A_2^*)<1$, for
every $z_1,z_2$ in $\mathbb T$ and hence
\[
|z_1\langle A_1^*\nu,\nu \rangle +z_2 \langle A_2^*\nu,\nu
\rangle|<1, \text{ for every } z_1,z_2 \in \mathbb T.
\]
If both $ \langle A_1^*\nu,\nu \rangle $ and $\langle A_2^*\nu,\nu
\rangle$ are non-zero, we can choose $z_1=\frac{|\langle
A_1^*\nu,\nu \rangle|}{\langle A_1^*\nu,\nu \rangle}$ and
$z_2=\frac{|\langle A_2^*\nu,\nu \rangle|}{\langle A_2^*\nu,\nu
\rangle}$ to get $|\alpha_1|+|\beta_1|<1$. If any of them or both
$ \langle A_1^*\nu,\nu \rangle $ and $\langle A_2^*\nu,\nu
\rangle$ are zero then also $|\alpha_1|+|\beta_1|<1$. Therefore,
by Theorem \ref{thm1}, $(x_1,x_2,x_3)$ is in $\mathbb E$. Thus,
$\Omega$ is non-empty and it exits through the distinguished
boundary $b\mathbb E$.

Again for any $x_3$, there is a unitary matrix $U$ of order $n$
(see Lemma \ref{spectra1}) such that $U^*(A_1^*+x_3A_2)U$ and
$U^*(A_2^*+x_3A_1)U$ have the following upper triangular form

\begin{align*}
U^*(A_1^*+x_3A_2)U &=
\begin{pmatrix}
\alpha_1+\bar{\beta_1}x_3 & \ast & \ast & \ast \\
0&\alpha_2+\bar{\beta_2}x_3 & \ast & \ast \\
\vdots & \vdots &\ddots & \vdots \\
0 & 0 & \cdots & \alpha_n+\bar{\beta_n}x_3
\end{pmatrix}\,, \\
U^*(A_2^*+x_3A_1)U &=
\begin{pmatrix}
\beta_1+\bar{\alpha_1}x_3 & \ast & \ast & \ast \\
0&\beta_2+\bar{\alpha_2}x_3 & \ast & \ast \\
\vdots & \vdots &\ddots & \vdots \\
0 & 0 & \cdots & \beta_n+\bar{\alpha_n}x_3
\end{pmatrix}
\end{align*}
and the joint spectrum $\sigma_T(A_1^*+x_3A_2,A_2^*+x_3A_1)$ can
be read off the diagonal of the common triangular form. It is
evident from definition that $\Omega$ has dimension one. Thus it
remains to show that $\Omega$ is a variety in $\mathbb E$. We show
that $\Omega$ is a variety in $\mathbb E$ determined by the ideal
generated by the set of polynomials
\[
\mathcal F
=\{\det[z_1(A_1^*+x_3A_2-x_1I)+z_2(A_2^*+x_3A_1-x_2I)]=0 \,:\,
z_1,z_2\in\overline{\mathbb D}\}.
\]
This is same as showing that ${\mathbb E}\cap \mathbb Z(\mathcal
F)=\Omega$, $\mathbb Z(\mathcal F)$ being the variety determined
by the ideal generated by $\mathcal F$. Let
$(x_1,x_2,x_3)\in\Omega$. Then $x_1={\alpha_k}+\bar{\beta_k}x_3$
and $x_2={\beta_k}+\bar{\alpha_k}x_3$ for some $k$ between $1$ and
$n$. Therefore,
$z_1({\alpha_k}+\bar{\beta_k}x_3-x_1)+z_2({\beta_k}+\bar{\alpha_k}x_3-x_2)=0$
for any $z_1,z_2$ in $\overline{\mathbb D}$ and consequently
$(x_1,x_2,x_3)\in\mathbb Z(\mathcal F)\cap \mathbb E$. Again let
$(x_1,x_2,x_3)\in\mathbb Z(\mathcal F)\cap \mathbb E$. Then
$\det[z_1(A_1^*+x_3A_2-x_1I)+z_2(A_2^*+x_3A_1-x_2I)]=0$ for all
$z_1,z_2\in\overline{\mathbb D}$ which implies that the two
matrices $A_1^*+x_3A_2-x_1I$ and $A_2^*+x_3A_1-x_2I$ have $0$ at a
common position in their diagonals. Thus $(x_1,x_2)$ is a joint
eigenvalue of $A_1^*+x_3A_2$ and $A_2^*+x_3A_1$ and
$(x_1,x_2,x_3)\in\Omega$. Hence $\Omega=\mathbb Z(\mathcal F)\cap
\mathbb E$ and $\Omega$ is a distinguished variety in $\mathbb E$.
The ideal generated by $\mathcal F$ must have a finite set of
generators and we leave it to an interested rader to determine
such a finite set.\\

Conversely, let $\Omega$ be a distinguished variety in $\mathbb
E$. We first show that $\Omega$ cannot be a two-dimensional
complex algebraic variety. Let if possible $\Omega$ be
two-dimensional and determined by a single polynomial $p$ in three
variables, i.e,
\[
\Omega =\{ (x_1,x_2,x_3)\in \mathbb E \,:\, p(x_1,x_2,x_3)=0 \}.
\]
Let $(y_1,y_2,y_3)\in\Omega$. Therefore, $|y_3|<1$. We show that
$\overline{\Omega}$ has intersection with $\partial \mathbb E
\setminus b\mathbb E$ which proves that $\Omega$ does not exit
through the distinguished boundary. Let $S_{y_3}$ be the set of
all points in $\Omega$ with $y_3$ as the third co-ordinate, i.e,
$S_{y_3}=\{ (x_1,x_2,x_3)\in\Omega\,:\,x_3=y_3 \}$. Such
$(x_1,x_2)$ are the zeros of the polynomial $p(x_1,x_2,y_3)$. If
$x_1=x_2$ for every $(x_1,x_2,y_3)\in S_{y_3}$, then
$p(x_1,x_2,y_3)$ becomes a polynomial in one variables and
consequently $S_{y_3}$ is a finite set. If every such $S_{y_3}$ is
a finite set then $\Omega$ becomes a one dimensional variety, a
contradiction. Therefore, there exists $y_3$ such that
$p(x_1,x_2,y_3)$ gives a one-dimensional variety and consequently
$S_{y_3}$ is not a finite set. We choose such $y_3$. Since each
$(x_1,x_2,y_3)$ in $S_{y_3}$ is a point in $\mathbb E$, by Theorem
\ref{thm1}, there exist complex numbers $\beta_1,\beta_2$ with
$|\beta_1|+|\beta_2|< 1$ such that $x_1=\beta_1+\bar{\beta_2}y_3$
and $x_2=\beta_2+\bar{\beta_1}y_3$. Let us consider the domain $G$
defined by
\[
G=\{(\beta_1,\beta_2)\in\mathbb C^2 \,:\, |\beta_1|+|\beta_2|<1
\},
\]
and the map
\begin{align*}
\varpi\,:\, \mathbb C^2 &\rightarrow \mathbb C^2 \\
(\beta_1,\beta_2) &\mapsto
(\beta_1+\bar{\beta_2}y_3,\beta_2+\bar{\beta_1}y_3).
\end{align*}
It is evident that the points $(x_1,x_2)$ for which
$(x_1,x_2,y_3)\in S_{y_3}$ lie inside $\varpi (G)$. Also it is
clear that $\varpi$ maps $G$ into $\mathbb D^2$ because the
tetrablock lives inside $\mathbb D^3$. This map $\varpi$ is
real-linear and invertible when considered a map from $\mathbb
R^4$ to $\mathbb R^4$, in fact a homeomorphism of $R^4$.
Therefore, $\varpi$ is open and it maps the boundary of $G$ onto
the boundary of $\varpi(G)$. Therefore, the zero-set of the
polynomial $p(x_1,x_2,y_3)$ in two variables ($y_3$ being a
constant) is a one-dimensional variety a part of which lies inside
$\varpi(G)$. Therefore, this variety intersects the boundary of
the domain $\varpi(G)$ which is the image of the set
$\{(\beta_1,\beta_2)\in\mathbb C^2 \,:\, |\beta_1|+|\beta_2|=1 \}$
under $\varpi$. Thus, there is a point $(\lambda_1,\lambda_2,y_3)$
in the zero set of $p$ such that
$\lambda_1=\beta_1+\bar{\beta_2}y_3$ and
$\lambda_2=\beta_2+\bar{\beta_1}y_3$ with $|\beta_1|+|\beta_2|=1$.
Therefore, $(\lambda_1,\lambda_2,y_3)\in \overline{\Omega}\cap
\overline{\mathbb E}$. Since $|y_3|<1$,
$(\lambda_1,\lambda_2,y_3)\in
\partial \mathbb E \setminus b\mathbb E$ and consequently $\Omega$
is not a distinguished variety, a contradiction. Thus, there is no
two-dimensional distinguished variety in $\mathbb E$ and $\Omega$
is one-dimensional.

Let $p_1,\hdots,p_n$, $(n>1)$ be polynomials in three variables
such that
\[
\Omega=\{ (x_1,x_2,x_3)\in\mathbb E \,:\,
p_1(x_1,x_2,x_3)=\hdots=p_n(x_1,x_2,x_3)=0  \}.
\]
We claim that all $p_i$ cannot be divisible by $x_3$. Indeed, if
$p_i$ is divisible by $x_3$ for all $i$ then $p_1=\hdots=p_n=0$
when $x_3=0$. The point $(0,1,0)\in \overline{\mathbb E}$ (by
choosing $\beta_1=0, \beta_2=1$ and applying Theorem \ref{thm1})
and clearly $p_i(0,1,0)=0$ for each $i$ but $(0,1,0) \notin
b\mathbb E$ although $(0,1,0)\in
\partial \mathbb E$ as $|\beta_1|+|\beta_2|=1$. This leads to the
conclusion that $\Omega$ does not exit through the distinguished
boundary of $\mathbb E$, a contradiction. Therefore, all
$p_1,\hdots,p_n$ are not divisible by $x_3$. Let $p_1$ be not
divisible by $x_3$ and
\begin{equation}\label{thm01}
p_1(x_1,x_2,x_3)= \sum_{ \substack{ 0\leq i \leq m_1\\
0\leq j \leq m_2 } } a_{ij}x_1^ix_2^j+x_3r(x_1,x_2,x_3)\,,
\end{equation}
for some polynomial $r$ and $a_{m_1m_2}\neq 0$.

Let $(M_{x_1},M_{x_2},M_{x_3})$ be the triple of operators on
$H^2(\mu)$ given by the multiplication by co-ordinate functions,
where $\mu$ is the measure as in Lemma \ref{basiclem2}. Then by
Lemma \ref{lempure}, $(M_{x_1},M_{x_2},M_{x_3})$ is a pure
$\mathbb E$-isometry on $H^2(\mu)$. Now $M_{x_3}M_{x_3}^*$ is a
projection onto $Ran\, M_{x_3}$ and
\[
Ran\,M_{x_3} \supseteq \{ x_3f(x_1,x_2,x_3):\; f \text{ is a
polynomial in } x_1,x_2,x_3 \}.
\]
It is evident from (\ref{thm01}) that
\[
a_{lk}x_1^lx_2^k \in \overline{Ran}\, M_{x_3}\oplus
\overline{\text{span}} \{ x_1^ix_2^j\,:\, 1\leq i \leq m_1,\,1\leq
j \leq m_2, i\neq l, j\neq k \},
\]
for each $k,l$ and hence
\[
H^2(\mu)=\overline{Ran}\, M_{x_3}\oplus \overline{\text{span}} \{
x_1^ix_2^j\,:\, 1\leq i\leq m_1, 1\leq j \leq m_2 \}.
\]
Therefore, $Ran\,(I-M_{x_3}M_{x_3}^*)$ has finite dimension, say
$n$. By Theorem \ref{model1}, $(M_{x_1},M_{x_2},M_{x_3})$ can be
identified with $(T_{\varphi},T_{\psi},T_{z})$ on $H^2(\mathcal
D_{M_{T_3}^*})$, where $\varphi(z)=A_1^*+A_2z$ and
$\psi(z)=A_2^*+A_1z$, $A_1,A_2$ being the fundamental operators of
$(T_{A_1^*+A_2z}^*,T_{A_2^*+A_1z}^*,T_{z}^*)$.
 By Lemma \ref{lemeval}, a point $(y_1,y_2,y_3)$ is in $\Omega$ if
 and only if $(\bar y_1, \bar y_2, \bar y_3)$ is a joint
 eigenvalue of $T_{\varphi}^*,T_{\psi}^*$ and $T_{z}^*$.
This can happen if and only if $(\bar y_1,\bar y_2)$ is a joint
eigenvalue of $\varphi(y_3)^*$ and $\psi(y_3)^*$. This leads to
\[
\Omega =\{ (x_1,x_2,x_3)\in \mathbb E\,:\, (x_1,x_2) \in \sigma_T(
A_1^*+x_3A_2, A_2^*+x_3A_1) \}.
\]
By the commutativity of $T_\varphi$ and $T_\psi$ we have that
$[A_1,A_2]=0$ and that $[A_1^*,A_1]=[A_2^*,A_2]$. The proof is now
complete.

\end{proof}

A variety given by the determinantal representation (\ref{eq:W}),
where $A_1,A_2$ satisfy $\|A_1^*+A_2z\|_{\infty, \mathbb T}=1$,
may or may not be a distinguished variety in the tetrablock as the
following examples illustrate.
\begin{eg}
Let us consider the commuting self-adjoint matrices
\[
A = \begin{pmatrix}
0 & 0 & 0 \\
0 & 0 & 1 \\
0 & 1 & 0
\end{pmatrix}
\textup{ and } B = \begin{pmatrix}
1 & 0 & 0 \\
0 & 0 & 0 \\
0 & 0 & 0
\end{pmatrix} .
\]
Then for any $z$ of unit modulus
\[
A+Bz = \begin{pmatrix}
z & 0 & 0 \\
0 & 0 & 1 \\
0 & 1 & 0
\end{pmatrix}
\]
is a normal matrix and for a unit vector
\[
h=\begin{pmatrix}
\alpha_1 \\
\alpha_2 \\
\alpha_3
\end{pmatrix} \in \mathbb C^3,
\]
 we have that $ \| (A+Bz)h \|=\| h \| $ and therefore, $\| A+Bz
\|=\omega(A+Bz)=1$. Now we define
\[
\Omega= \{ (x_1,x_2,x_3)\in \mathbb E\,:\,
(x_1,x_2)\in\sigma_T(A+x_3B,B+x_3A)  \}.
\]
Clearly $(1,0,0)\in \partial \mathbb E \cap \overline{\Omega}$, by
Theorem \ref{thm1} (by choosing $x_3=0, \beta_1=1$ and
$\beta_2=0$) but $(1,0,0) \notin b\mathbb E$ which shows that
$\Omega$ does not exit through the distinguished boundary
$b\mathbb E$. Hence $\Omega$ is not a distinguished variety.
\end{eg}

\begin{eg}
Let
\[
A_1 = A_2 = \begin{pmatrix}
0 & 1 & 0 \\
0 & 0 & 0 \\
0 & 0 & 0
\end{pmatrix}.
\]
Then
\[
A_1^*+A_2z = A_1^*+A_1z= \begin{pmatrix}
0 & z & 0 \\
1 & 0 & 0 \\
0 & 0 & 0
\end{pmatrix}
\]
and $\| A_1^*+A_1z \|=1$, for all $z\in\mathbb T$. Let
\[
\Omega =\{ (x_1,x_2,x_3)\in\mathbb E \,:\,
(x_1,x_2)\in\sigma_T(A_1^*+x_3A_1,A_1^*+x_3A_1)\}
\]
Here
\[
A_1^*+x_3A_1-xI = \begin{pmatrix}
-x & x_3 & 0 \\
1 & -x & 0 \\
0 & 0 & -x
\end{pmatrix}
\]
and thus
\[
\det (A_1^*+x_3A_1-xI)=x(x_3-x^2).
\]
Clearly $\Omega$ is non-empty as it contains the points of the
form $(0,0,x_3)$. Clearly this sheet of the variety $\Omega$ exits
through $b\mathbb E$. It is evident that $\Omega$ is
one-dimensional. Also, $\Omega$ contains the points $(x,x,x_3)$
with $x^2=x_3$. By Theorem \ref{thm1}, we have that
$x=\beta_1+\bar{\beta_2}x_3=\beta_2+\bar{\beta_1}x_3$, for some
$\beta_1,\beta_2$ with $|\beta_1|+|\beta_2|\leq 1$. Now when
$x_3\neq 0$, $x\neq 0$ and hence $(\beta_1,\beta_2)\neq (0,0)$.
When $\beta_1 \neq \beta_2$, we have
\[
|x_3|=\left| \frac{\beta_1-\beta_2}{\bar{\beta_1}-\bar{\beta_2}}
\right| =1
\]
and hence $(x,x,x_3)\in b\mathbb E$. When $\beta_1=\beta_2=\beta$,
we show that $(x,x,x_3)\in\mathbb E$ if $|x_3|<1$. Let $|x_3|<1$
and $x=\beta+ \bar{\beta} x_3$. It suffices to show that
$|\beta|+|\bar{\beta}|<1$, i.e, $|\beta|<1/2$. Let if possible
$|\beta|=1/2$ and $\beta=\frac{1}{2}e^{i\theta}$. Since $x^2=x_3$,
without loss of generality let $x=\sqrt{x_3}$. So we have
\[
\sqrt{x_3}=\beta +\bar{\beta} x_3 =\frac{1}{2}e^{i\theta}+
\frac{1}{2}e^{-i\theta}x_3,
\]
which implies that $ (\sqrt{x_3}-e^{i\theta})^2=0 $. Therefore,
$|x_3|=1$, a contradiction. Thus $|\beta|<1/2$ and $(x,x,x_3)\in
\mathbb E$. Therefore, $\overline{\Omega}\cap \partial \mathbb
E=\overline{\Omega}\cap b\mathbb E$ and $\Omega$ is a
distinguished variety.
\end{eg}

Below we characterize all distinguished varieties for which
$\|A_1+A_2z\|_{\infty,\mathbb T}<1$.

\begin{thm}\label{char:1}
Let $\Omega$ be a variety in $\mathbb E$. Then
\[
\Omega = \{(x_1,x_2,x_3) \in \mathbb E \,:\,
(x_1,x_2)\in\sigma_T(A_1^*+x_3A_2, A_2^*+x_3A_1)\}
\]
for two commuting square matrices $A_1,A_2$ satisfying
$[A_1^*,A_1]=[A_2^*,A_2]$ and $\|A_1+A_2z\|_{\infty,\mathbb T}<1$
if and only if $\Omega$ is a distinguished variety in $\mathbb E$
such that $\partial \Omega \cap bD_{\mathbb E} = \emptyset$, where
\[
bD_{\mathbb E} = \{(x_1,x_2,x_1x_2) \,:\, |x_1|=|x_2|=1 \}.
\]

\end{thm}

\begin{proof}

Recall that
\begin{align*}
b\mathbb E & = \{ (x_1,x_2,x_3)\in\mathbb C^3\,:\, x_1=\bar
x_2x_3,\,|x_3|=1 \textup{ and } |x_2|\leq 1 \}
\\& = \{ (x_1,x_2,x_3)\in\overline{\mathbb E}: |x_3|=1 \}.
\end{align*}
It is clear that $bD_{\mathbb E} \subseteq b\mathbb E $. If
$\Omega$ has such an expression in terms of joint eigenvalues of
$A_1^*+A_2z$ and $A_2^*+A_1z$ then by Theorem \ref{thm:DVchar},
$\Omega$ is a distinguished variety in $\mathbb E$. We show that
$\partial \Omega \cap bD_{\mathbb E} = \emptyset$. Obviously
$A_1^*+A_2z$ and $A_2^*+A_1z$ are commuting normal matrices for
every $z\in \mathbb T$. If $(x_1,x_2, e^{i\theta}) \in
\partial \Omega$, then $(x_1,x_2)$ is a joint eigenvalue of
$A_1^*+e^{i\theta}A_2$ and $A_2^*+e^{i\theta}A_1$. But
$\|A_1^*+A_2z\| = \omega(A_1^*+A_2z) <1$ and hence $|x_i|<1$ for
$i=1,2$. Therefore, $(x_1,x_2,e^{i\theta}) \notin bD_\mathbb E$.

Conversely, suppose that $\Omega$ is a distinguished variety such
that $\partial \Omega \cap bD_{\mathbb E} = \emptyset$. In course
of the proof of Theorem \ref{thm:DVchar} we showed that $\Omega$
is given by (\ref{eq:W}) with $A_1,A_2$ being the fundamental
operators of $(M_{x_1}^*,M_{x_2}^*,M_{x_3}^*)$ on $H^2(\mu)$. What
we need to show is that $\|A_1+A_2z\|_{\infty,\mathbb T}<1$.

We saw in the proof of the Theorem \ref{thm:DVchar} that
$(M_{x_1},M_{x_2},M_{x_3})$ is unitarily equivalent to
$(T_\varphi,T_\psi,T_z)$ on $H^2(\mathcal D_{M_{T_3}^*})$, where
$\varphi(z) = A_1^* + A_2z$ and $ \psi(z)=A_2^*+A_1z$, for two
commuting matrices $A_1,A_2$ satisfying $[A_1^*,A_1]=[A_2^*,A_2]$.
Since $\partial \Omega \cap bD_{\mathbb E} = \emptyset$, we have
that $ \| M_{x_1}^* \|<1$ and $\| M_{x_2}^* \|<1 . $ Therefore,
\[
\| T_{\varphi} \|= \| A_1^*+A_2z \|_{\infty,\mathbb
T}=\sup_{z\in\mathbb T} \;\omega(A_1^*+A_2z)<1
\]
and the proof is complete.

\end{proof}

\section{A connection with the bidisc and the symmetrized bidisc}

\noindent Recall that the symmetrized bidisc $\mathbb G$, its
closure $\Gamma$ and the distinguished boundary $b\Gamma$ are the
following sets:
\begin{align*}
\mathbb G &= \{ (z_1+z_2,z_1z_2)\,:\,|z_1|< 1, |z_2|< 1
\}\subseteq \mathbb C^2;\\ \Gamma &= \{
(z_1+z_2,z_1z_2)\,:\,|z_1|\leq 1, |z_2|\leq 1 \};\\ b\Gamma &= \{
(z_1+z_2,z_1z_2)\,:\,|z_1|= 1, |z_2|= 1 \}\\ \quad &= \{
(s,p)\in\Gamma\,:\,|p|= 1 \}.
\end{align*}
A pair of commuting operators $(S,P)$ on a Hilbert space $\mathcal
H$ that has $\Gamma$ as a spectral set, is called a
$\Gamma$-contraction. The symmetrized bidisc enjoys rich operator
theory \cite{ay-jfa,ay-jot,tirtha-sourav,tirtha-sourav1, sourav}.
Operator theory and complex geometry of the tetrablock have
beautiful connections with that of the symmetrized bidisc as was
shown in \cite{tirtha}. We state here two of the important results
in this line from \cite{tirtha}.

\begin{lem}\label{tirtha-thm1}
A point $(x_1,x_2,x_3)\in \overline{\mathbb E}$ if and only if
$(x_1+zx_2,zx_3)\in \Gamma $ for every $z$ on the unit circle.
\end{lem}
\begin{thm}\label{tirtha-thm2}
Let $(T_1,T_2,T_3)$ be an $\mathbb E$-contraction. Then
$(T_1+zT_2,zT_3)$ is a $\Gamma$-contraction for every $z$ of unit
modulus.
\end{thm}

See Lemma 3.2 and Theorem 3.5 respectively in \cite{tirtha} for
details. The distinguished varieties in the symmetrized bidisc,
their representations and relations with the operator theory have
been described beautifully in \cite{pal-shalit}. We recall from
\cite{pal-shalit} Lemma 3.1 and Theorem 3.5 which will help
proving the main result of this section, Theorem \ref{connection}.
\begin{lem}\label{lem:DVA}
Let $W\subseteq \mathbb G$. Then $W$ is a distinguished variety in
$\mathbb G$ if and only if there is a distinguished variety $V$ in
$\mathbb{D}^2$ such that $W=\pi(V)$, where $\pi$ is the
symmetrization map from $\mathbb C^2$ to $\mathbb C^2$ that maps
$(z_1,z_2)$ to $(z_1+z_2,z_1z_2)$.
\end{lem}

\begin{thm}\label{thm:DVsym}
Let $A$ be a square matrix A with $\omega(A)< 1$, and let $W$ be
the subset of $\mathbb G$ defined by
\begin{equation*}
W = \{(s,p) \in \mathbb G \,:\, \det(A + pA^* - sI) = 0\}.
\end{equation*}
Then $W$ is a distinguished variety. Conversely, every
distinguished variety in $\mathbb G$ has the form $\{(s,p) \in
\mathbb G \,:\, \det(A + pA^* - sI) = 0\}$, for some matrix $A$
with $\omega(A)\leq 1$.
\end{thm}

Let us consider the holomorphic map whose source is Lemma
\ref{tirtha-thm1}:
\begin{align*} & \phi\,:\,
\overline{\mathbb E} \longrightarrow \Gamma \\& (x_1,x_2,x_3)
\mapsto (x_1+x_2,x_3).
\end{align*}

\begin{thm}\label{connection}
If
\[
\Omega=\{ (x_1,x_2,x_3)\in\mathbb E
\,:\,(x_1,x_2)\in\sigma_T(A_1^*+A_2x_3,A_2^*+A_1x_3) \}
\]
is a distinguished variety in $\mathbb E$, then $W=\phi (\Omega)$
is a distinguished variety in $\mathbb G$ provided that
$\|A_1^*+A_2z\|_{\infty, \mathbb T}<1$. Moreover, $\Omega$ gives
rise to a distinguished variety in $\mathbb D^2$.
\end{thm}
\begin{proof}
Clearly
\[
W=\{ (x_1+x_2,x_3)\,:\,(x_1,x_2,x_3)\in\Omega \}.
\]
Since $(x_1,x_2)\in\sigma_T(A_1^*+A_2x_3,A_2^*+A_1x_3)$, $x_1+x_2$
is an eigenvalue of $(A_1+A_2)^*+x_3(A_1+A_2)$. Therefore,
\begin{align*}
W &=\{ (x_1+x_2,x_3)\in\mathbb G\,:\, \det
[(A_1+A_2)^*+x_3(A_1+A_2)-(x_1+x_2)I]=0 \}\\
&=\{ (s,p)\in\mathbb G\,:\, \det [(A_1+A_2)^*+p(A_1+A_2)-sI]=0 \},
\end{align*}
where $\omega(A_1+A_2)<1$, by Lemma \ref{funda-properties}.
Therefore, by Theorem \ref{thm:DVsym}, $W$ is a distinguished
variety in $\mathbb G$. Also, the existence of a distinguished
variety $V$ in $\mathbb D^2$ with $\pi(V)=W$ is guaranteed by
Lemma \ref{lem:DVA}.

\end{proof}
It is still unknown whether the other way is also true, i.e,
whether every distinguished variety in $\mathbb G$ or in $\mathbb
D^2$ gives rise to a distinguished variety in $\mathbb E$. Our
wild guess to this question is in the negative direction and the
reason is that every distinguished variety in the symmetrized
bidisc has representation in terms of the fundamental operator of
a $\Gamma$-contraction as was shown in \cite{pal-shalit} and it is
still unknown whether every $\Gamma$-contraction gives rise to an
$\mathbb E$-contraction although the other way is true according
to Theorem \ref{tirtha-thm2}.

\section{A von-Neumann type inequality for $\mathbb E$-contractions}

\begin{thm}\label{thm:VN}
Let $\Upsilon=(T_1,T_2,T_3)$ be an $\mathbb E$-contraction on a
Hilbert space $\mathcal H$ such that $(T_1^*,T_2^*,T_3^*)$ is a
pure $\mathbb E$-contraction and that $\dim \mathcal D_{T_3} <
\infty$. Suppose that the fundamental operators $A_1,A_2$ of
$(T_1,T_2,T_3)$ satisfy $[A_1,A_2]=0$ and
$[A_1^*,A_1]=[A_2^*,A_2]$. If
\[
\Omega_{\Upsilon} = \{(x_1,x_2,x_3) \in \overline{\mathbb E} :
(x_1,x_2)\in \sigma_T(A_1^*+x_3A_2,A_2^*+x_3A_1)\},
\]
then for every scalar or matrix-valued polynomial $p$ in three
variables,
\[
\|p(T_1,T_2,T_3)\| \leq \max_{(x_1,x_2,x_3) \in
\Omega_{\Upsilon}\cap b\mathbb E} \|p(x_1,x_2,x_3)\|.
\]
Moreover, when $\omega(A_1+A_2z)<1$ for every $z$ of unit modulus,
$\Omega_\Upsilon \cap \mathbb E$ is a distinguished variety in the
tetrablock.
\end{thm}

\begin{proof}

Suppose that $\dim \mathcal D_{T_3}=n $ and then $A_1,A_2$ are
commuting matrices of order $n$. We apply Theorem
\ref{dilation-theorem} to the pure $\mathbb E$-contraction
$(T_1^*,T_2^*,T_3^*)$ to get an $\mathbb E$-co-isometric extension
on $H^2(\mathcal D_{T_3})$ of $(T_1,T_2,T_3)$. Therefore,
\[
T_{A_1^*+A_2z}^*|_{\mathcal H}=T_1,\;T_{A_2^*+A_1z}^*|_{\mathcal
H}=T_2, \text{ and } T_z^*|_{\mathcal H}=T_3.
\]
Let $\varphi$ and $\psi$ denote the $\mathcal L(\mathcal D_{T_3})$
valued functions $\varphi(z)=A_1^*+A_2z$ and $\psi(z)=A_2^*+A_1z$
respectively. Let $p$ be a scalar or matrix-valued polynomial in
three variables and let $p_*$ be the polynomial satisfying
$p_*(A,B)=p(A^*,B^*)^*$ for any two commuting operators $A,B$.
Then
\begin{align*}
\|p(T_1,T_2,T_3)\| &\leq \|p(T_{\varphi}^*,T_{\psi}^*, T_z^*)\|_{H^2(\mathcal D_{T_3})} \\
&= \|p_*(T_{\varphi},T_{\varphi}, T_z)\|_{H^2(\mathcal D_{T_3})}
\\& \leq \|p_*(M_{\varphi},M_{\varphi},M_z)\|_{L^2( \mathcal D_{T_3})} \\
& = \max_{\theta \in [0,2\pi]} \|p_*
(\varphi(e^{i\theta}),\psi(e^{i\theta}), e^{i\theta}I)\|.
\end{align*}
It is obvious from Theorem \ref{tu} that
$(M_{\varphi},M_{\varphi},M_z)$ is an $\mathbb E$-unitary as $M_z$
on $L^2(\mathcal D_{T_3})$ is a unitary. Therefore, $M_{\varphi}$
and $M_{\psi}$ are commuting normal operators and hence
$\varphi(z)$ and $\psi(z)$ are commuting normal operators for
every $z$ of unit modulus. Therefore,
\[
\|p_*(\varphi(e^{i\theta}),\psi(e^{i\theta}), e^{i\theta}I)\|=
\max \{ |p_*(\lambda_1,\lambda_2,e^{i\theta})|:\,
(\lambda_1,\lambda_2)\in
\sigma_T(\varphi(e^{i\theta}),\psi(e^{i\theta}) \}.
\]
Let us define
\[
\Omega_\Upsilon=\{ (x_1,x_2,x_3)\in\overline{\mathbb E}:\;
(x_1,x_2)\in \sigma_T(A_1^*+x_3A_2,A_2^*+x_3A_1) \}
\]
and
\begin{align*}
\Omega_\Upsilon^* &=\{ (x_1,x_2,x_3)\in\overline{\mathbb E}:\;
(x_1,x_2)\in \sigma_T(A_1+x_3A_2^*,A_2+x_3A_1^*) \} \\& =\{
(\bar{x_1},\bar{x_2},\bar{x_3})\in\overline{\mathbb E}:\;
(x_1,x_2)\in \sigma_T(A_1^*+x_3A_2,A_2^*+x_3A_1) \}.
\end{align*}
It is obvious that both $\Omega_\Upsilon$ and $\Omega_\Upsilon^*$
are one-dimensional subvarieties of $\overline{\mathbb E}$. We now
show that if $(\lambda_1,\lambda_2)\in
\sigma_T(\varphi(e^{i\theta}),\psi(e^{i\theta}))$ then
$(\lambda_1,\lambda_2,e^{i\theta})$ is in $\overline{\mathbb E}$.
There exists a unit vector $\nu$ such that
\[
(A_1^*+e^{i\theta}A_2)\nu=\lambda_1\nu \text{ and }
(A_2^*+e^{i\theta}A_1)\nu=\lambda_2\nu.
\]
Taking inner product with $\nu$ we get $
\beta_1+\bar{\beta_2}e^{i\theta}=\lambda_1 \textup{ and }
\beta_2+\bar{\beta_1}e^{i\theta}=\lambda_2 , $ where
$\beta_1=\langle A_1^*\nu,\nu \rangle$ and $\beta_2=\langle
A_2^*\nu,\nu \rangle$. Since $A_1,A_2$ are the fundamental
operators, by Lemma \ref{funda-properties}, $\omega(A_2^*+zA_1)$
is not greater than $1$ for every $z$ of unit modulus. Therefore,
\[
|\lambda_2|=|\beta_1+\bar{\beta_2}e^{i\theta}| \leq 1.
\]
Again
\[
\bar{\lambda}_2e^{i\theta}=\beta_1+\bar{\beta_2}e^{i\theta}=\lambda_1.
\]
Therefore by (\ref{dv1}), $(\lambda_1,\lambda_2,e^{i\theta})\in
b\mathbb E \subseteq \overline{\mathbb E}$. Therefore, we conclude
that
\begin{align*}
\| p(T_1,T_2,T_3) \| &\leq \max_{(x_1,x_2,x_3)\in
\Omega_\Upsilon^* \cap b\mathbb E} \| p_*(x_1,x_2,x_3) \|
\\&=\max_{(x_1,x_2,x_3)\in \Omega_\Upsilon \cap b\mathbb E} \|
p(x_1,x_2,x_3) \|.
\end{align*}
If $\omega(A_1+A_2z)<1$ for all $z\in\mathbb T$, by Lemma
\ref{funda-properties},
$\omega(A_1^*+A_2z)=\|A_1^*+A_2z\|_{\infty, \mathbb T}<1$. It is
now obvious from Theorem \ref{thm:DVchar} that $\Omega_\Upsilon
\cap \mathbb E$ is a distinguished variety in the tetrablock.

\end{proof}

We have established the fact that if an $\mathbb E$-contraction
$(T_1,T_2,T_3)$ satisfies the hypotheses of Theorem \ref{thm:VN}
then there is an $\mathbb E$-co-isometric extension of
$(T_1,T_2,T_3)$ that lives on the corresponding variety
$\Omega_\Upsilon$. This also makes $\Omega_\Upsilon ^*$ a complete
spectral set for $(T_1^*,T_2^*,T_3^*)$. Obviously this is valid
when $(T_1,T_2,T_3)$ consists of matrices and $\| T_3 \|<1$. We do
not know whether there is a bigger class of $\mathbb
E$-contractions for which von-Neumann type inequality is valid on
such a one-dimensional subvariety. This line of proof of Theorem
\ref{thm:VN} will no longer be valid if there is a bigger class
and also $\Omega_\Upsilon ^*$ will not be a complete spectral set
for $(T_1^*,T_2^*,T_3^*)$ in that case as we have used the
functional model of $(T_1^*,T_2^*,T_3^*)$ in the proof.

\vspace{0.62cm}

\noindent \textbf{Acknowledgement.} A part of this work was done
when the author was visiting Ben-Gurion University of the Negev,
Israel and was funded by the Skirball Foundation via the Center
for Advanced Studies in Mathematics at Ben-Gurion University of
the Negev, Israel. The author conveys his sincere thanks to Orr
Shalit who made numerous invaluable comments on the first draft of
this article. Also in course of the work the author visited Indian
Statistical Institute, Delhi Centre and greatly appreciates the
warm and generous hospitality provided there. Finally the author
thanks his colleagues in Indian Institute of Technology Bombay for
creating an amazing research atmosphere where the work was
completed.

\end{document}